\documentclass[pdflatex,sn-mathphys-num]{sn-jnl}

\usepackage{graphicx}%
\usepackage{multirow}%
\usepackage{amsmath,amssymb,amsfonts}%
\usepackage{amsthm}%
\usepackage{mathrsfs}%
\usepackage[title]{appendix}%
\usepackage{xcolor}%
\usepackage{textcomp}%
\usepackage{manyfoot}%
\usepackage{booktabs}%
\usepackage{algorithm}%
\usepackage{algorithmicx}%
\usepackage{algpseudocode}%
\usepackage{listings}%

  \theoremstyle{plain}
  \newtheorem{thm}{Theorem}[section]
  \theoremstyle{plain}
  
  \theoremstyle{plain}
  \newtheorem{prop}[thm]{Proposition}
  \theoremstyle{remark}
  \newtheorem{rem}[thm]{Remark}
  \theoremstyle{remark}
  
  \theoremstyle{plain}
  \newtheorem{lem}[thm]{Lemma}
  
    \newtheorem{ex}[thm]{Example}

\renewcommand{\thetable}{\arabic{table}}
\renewcommand{\thefigure}{\arabic{figure}}

\smallskip
\def\<{{\langle }}
\def\>{{\rangle }}

\smallskip
\def\<{{\langle }}
\def\>{{\rangle }}

\def\bar#1{\overline{#1}}

\theoremstyle{thmstyleone}%
\theoremstyle{thmstyletwo}%

\theoremstyle{thmstylethree}%

\begin{document}

\title[Article Title]{Spectrum of the Laplacian and the Jacobi operator on generalized rotational minimal hypersurfaces of spheres}


\author{\fnm{Oscar} \sur{Perdomo}}\email{perdomoosm@ccsu.edu}

\affil{\orgdiv{Department of Mathematics}, \orgname{Central Connecticut State University}, \orgaddress{\street{1650 Stanley St.}, \city{New Britain, CT}, \state{CT}, \country{USA}}}


\abstract{  Let $M\subset S^{n+1}$ be the hypersurface generated by rotating a hypersurface $M_0$ contained in the interior of the unit ball of $\mathbb{R}^{n-k+1}$. More precisely, $M=\{(\sqrt{1-|m|^2}\, y, m):y\in S^k, m\in M_0\}$. We derive the equation for the mean curvature of $M$ in terms of the principal curvatures of $M_0$. For the particular case when $M_0$ is a surface of revolution in $\mathbb{R}^3$, we provide a method for finding the eigenvalues of the Laplace and stability operators. To illustrate this method, we consider an example of a minimal embedded hypersurface in $S^6$ and numerically compute all the eigenvalues of the Laplace operator less than 12, as well as all non-positive eigenvalues of the stability operators. For this example, we show that the stability index (the number of negative eigenvalues of the stability operator, counted with multiplicity) is 77, and the nullity (the multiplicity of the eigenvalue $\lambda=0$ of the stability operator) is 14. Similar results are found in the case where $M_0$ is a hypersurface in $\mathbb{R}^{l+2}$ of the form $(f_2(u)z, f_1(u))$, with $z$ in the $l$-dimensional unit sphere $S^l$. Carlotto and Schulz have found examples of embedded minimal hypersurfaces in the case where $M_0=S^k\times S^1$.}


\keywords{Minimal hypersurfaces, spheres, Eigenvalues of the Laplacian, Stability index, nullity.}



\maketitle
\section{Introduction}

Minimal hypersurfaces in $S^{n+1}$ are critical points of the $n$-dimensional area functional. If $M\subset S^{n+1}$ is an oriented compact  minimal hypersurface and $\nu:M\longrightarrow S^{n+1}$ is a Gauss map, then for any  function $f:M\longrightarrow \mathbb{R}$, when we consider the 1-parametric family of hypersurface $M_t=\{\cos(t f(m)) m+\sin(t f(m))\nu(m): m\in M\}\subset S^{n+1}$, the function  $h(t)=n-area(M_t)$ satisfies:

$$h(0)=n-\hbox{area}(M), \quad h^\prime(0)=\int_M Hf \quad\hbox{and}\quad h^{\prime\prime}(0)=\int_MJ(f)f$$

where $H:M\longrightarrow \mathbb{R}$ is the mean curvature and $J(f)=-\Delta f-|A|^2f-nf$ is the Jacobi or stability operator. Here $|A|^2=\lambda_1^2+\dots +\lambda_n^2$ is the square of the norm of the shape operator and $\lambda_i$ are the principal curvatures of $M$. 

These are the some known examples of minimal hypersurfaces of spheres:
\begin{itemize}
 \item
 The totally umbilical example: $S^n=\{x\in S^{n+1}:x_{n+2}=0\}$.
 
 \item
  The minimal Clifford hypersurfaces $M=\{(y,z):y\in\mathbb{R}^{k+1}, \, z\in\mathbb{R}^{n-k+1},\, |y|^2=\frac{k}{n},\, |z|^2=\frac{n-k}{n} \}$.
 \item
 
 The isoparametric minimal hypersurfaces: These examples can be considered as generalizations of the Clifford hypersurfaces, and they are characterized by having constant principal curvatures, see  \cite{M1} and \cite{M2}.

 \item
 
Rotational minimal hypersurfaces: Another family that can be considered a generalization of the Clifford hypersurface. They are characterized by the property that at every point, there are exactly two principal curvatures (this time, not necessarily constant). It is known that the only embedded examples are the Clifford hypersurfaces; see \cite{O1} and \cite{O2}.
 
 \item
Homogeneous hypersurfaces with low cohomogeneity: For these examples, the group of isometries is big enough so that the hypersurface can be described by a curve. In some particular cases, the ordinary differential equations (ODEs) for these curves were solved, and therefore explicit examples were presented, see \cite{HL}.
 
  \item
 Lawson examples: Lawson provided several families of minimal surfaces in $S^3$, \cite{L}. 
 
\item

Karcher-Pinkall-Sterling family of surfaces in $S^3$, see \cite{KPS}. 

\item

Kapouleas examples. Several families of minimal surfaces in $S^3$, see \cite{K1}, \cite{KW}, \cite{KY}, \cite{W}. 

\item

Carlotto-Schulz examples. They considered the examples that we are considering in this paper, hypersurfaces of the form $M=\{(\sqrt{1-|m|^2}\, y, m):y\in S^k, m\in M_0\}$, for the particular case that $M_0=S^k\times S^1$ and showed the existence of minimal embedded examples in the  unit dimensional sphere $S^{2k+1}$. 

 \end{itemize}

There are very interesting questions regarding the spectrum of the Laplace operator and the stability operator. The only minimal hypersurfaces for which the complete spectrum of the Laplacian is known are the totally geodesic spheres, the Clifford hypersurfaces, and the cubic isoparametric hypersurfaces,  \cite{S1}, \cite{S2}.  In \cite{P4}, the author found a way to numerically compute the spectra of the Laplacian and the stability operator for the rotational hypersurfaces. For general minimal hypersurfaces in the sphere, regarding the spectrum of the Laplacian, clearly $0$ is an eigenvalue, and it is known that if $M$ is embedded and compact with no boundary, then the first eigenvalue of the Laplacian, $\lambda_1(M)$, is greater than $n/2$; see \cite{CW}.  Recently, this bound on $\lambda_1(M)$ was improved in \cite{DSS}. It is known that $n$, the dimension of $M$, is an eigenvalue of the Laplace operator.  Yau's conjecture, \cite{Y}, states that if $M$ is compact and embedded, then there are no eigenvalues between $0$ and $n$, this is, the conjecture states that $n$ is  the first positive eigenvalue of $-\Delta$. 

This conjecture has been verified for the Clifford hypersurfaces; the cubic isoparametric; degree 4 isoparametric hypersurfaces, \cite{ZW}; several of the homogeneous hypersurfaces with low cohomogeneity, \cite{MOU}; the Lawson minimal surfaces and the Karcher-Pinkall-Sterling minimal surfaces in $S^3$, \cite{CS}; and, more recently, for a broader class of minimal surfaces in $S^3$ with reflectional symmetries by Kusner and McGrath, \cite{KusMcG}.

The stability index of $M$ is defined as the number of negative eigenvalues of $J$ counted with multiplicities. For the totally geodesic spheres, the stability index is $1$, and for the Clifford minimal hypersurfaces, the stability index is $n+3$. It has been conjectured that any other compact minimal hypersurface has a stability index greater than $n+3$. The conjecture has been proven only for surfaces, the case $n=2$, by Urbano, \cite{U}. In the case that $M$ has antipodal symmetry the conjecture was shown in \cite{P1}.   It is known that $0$ is an eigenvalue of the stability operator and its multiplicity is called its nullity. For one of the families of minimal surfaces introduced by Lawson, those denoted as $\xi_{g,1}$, Kapouleas and Wiygul showed that their stability index is $2g+3$ and the nullity is $6$.


In this paper, we will consider hypersurfaces of  $S^{n+1}$ that can be written in the form $(\sqrt{1-|m|^2}\, y, m)$, where the points $m$ are in a hypersurface $M_0$ of $\mathbb{R}^{n-k+1}$ and $y\in S^k$. We find the mean curvature of $M$ in terms of the principal curvatures of $M_0$. When $M_0$ is a hypersurface of revolution of the form $(f_2(u)z,f_1(u))$ with $z\in S^l$, we  compute the Laplace and the stability operator in terms of the two real functions $f_1(u)$ and $f_2(u)$ and we describe a method for finding  the spectrum of these two operators. Several numerical examples of embedded minimal hypersurfaces are found.  


\section{Computing the mean curvature}

Let $M_0\subset \mathbb{R}^{n-k+1}$ be a hypersurface satisfying  $|m|<1$ for all $m\in M_0$. Let us denote by $S^k$ the unit sphere  in $\mathbb{R}^{k+1}$. Let us consider the immersion $\varphi:S^k\times M_0\longrightarrow M\subset S^{n+1}\subset\mathbb{R}^{n+2}$ given by 

\begin{eqnarray}\label{theexamples}
\varphi(y,m)=(\sqrt{1-|m|^2}\, y, m). 
\end{eqnarray}

Let us assume that $M_0$ is orientable and that  $N:M_0\longrightarrow \mathbb{R}^{n-k+1}$ is a  Gauss map. A direct computation shows that a Gauss map (defined on $S^k\times M_0$, not on $M=\varphi(S^k\times M_0)$) for the immersion $\varphi$ is given by 

$$\xi(y,m)=\frac{1}{\sqrt{1-(N(m)\cdot m)^2}}\left( -(N(m)\cdot m) \sqrt{1-|m|^2}\, y\, ,\, -(N(m)\cdot m)\, m+N(m) \right). $$

Let us denote by $\kappa_i$ the principal curvatures of  $M_0$. More precisely, let us assume that $dN_m(v_i)=-\kappa_i v_i$ where $\{v_1,\dots, v_{n-k}\}$ forms an orthonormal basis of $T_mM_0$. If $w$ is a unit vector in $T_yS^k$ and $p=(y,m)$, then 

$$d\varphi_p(w)=(\sqrt{1-|m|^2}\, w, 0,\dots,0)$$

and 

$$d\xi_p(w)=\frac{ -N(m)\cdot m}{\sqrt{1-(N(m)\cdot m)^2}}\left( \sqrt{1-|m|^2}\, w\, , 0\dots 0\right).$$

Therefore 

\begin{eqnarray}\label{lambda0}
d\xi_p(w)=-\lambda_0 d\varphi_p(w)\, \quad \hbox{with} \quad \lambda_0=\frac{ N(m)\cdot m}{\sqrt{1-(N(m)\cdot m)^2}} .
\end{eqnarray}

We have that $\lambda_0$ is a principal curvature of $\varphi$ with multiplicity at least $k$. Before we compute the other $n-k$ principal curvatures of the immersion $\varphi$ let us consider the following lemma,

\begin{lem}\label{fgh} If $f=\sqrt{1-|m|^2}$, $g=N(m)\cdot m$ and $h=\sqrt{1-(N(m)\cdot m)^2}$ and $dN(v_i)=-\kappa_iv_i$, then

$$dg_m(v_i)=-\kappa_i v_i\cdot m,\quad df_m(v_i)=-\frac{v_i\cdot m}{f}\quad \hbox{and}\quad
dh_m(v_i)=\frac{g\kappa_i (v_i\cdot m)}{h} .$$

Moreover,

$$v_i(h^{-1})=-\frac{g(v_i\cdot m)\kappa_i}{h^3},\, \quad v_i(gh^{-1})=- \frac{\kappa_i(v_i\cdot m)}{h^3} \quad\hbox{and}
\quad v_i(fgh^{-1})=-\frac{v_i\cdot m}{h^3f}(gh^2+\kappa_i f^2) .$$
\end{lem}

\begin{proof} Since $N(m)\cdot v_i=0$, then $dg_m(v_i)=dN_m(v_i)\cdot m+N(m)\cdot v_i=-\kappa_i\, v_i\cdot m$. The other identities are similar. Let us check the last identity, 

\begin{eqnarray*}
v_i(fgh^{-1})&=& v_i(f)\frac{g}{h}+fv_i(gh^{-1})\\
  &=& -\frac{g(v_i\cdot m)}{fh}-\frac{f\kappa_i(v_i\cdot m)}{h^3} \\
  &=& -\frac{v_i\cdot m}{fh^3}\left(gh^2+\kappa_if^2 \right) .
  \end{eqnarray*}
\end{proof}

\begin{prop} \label{nH} Let $f,g$ and $h$ be the functions defined in Lemma \ref{fgh}. We have that

$$d\varphi(v_i)=\left(-\frac{v_i\cdot m}{f}\, y\, , \, v_i \right).$$

Moreover, if $H$ is the mean  curvature of $M$ and $H_0$ is the mean curvature of $M_0$ then, 

\begin{eqnarray}\label{eqH}
nH=\frac{ng+(n-k)H_0}{h}-\frac{1}{h^3}\sum_{j=1}^{n-k } \kappa_i(v_i\cdot m)^2 .
\end{eqnarray}

\end{prop}

\begin{proof}
We have that $-\xi=\left(gfh^{-1} \, y\, , \, gh^{-1}m-h^{-1} N\right)$. Using Lemma \ref{fgh}, we get that 

\begin{eqnarray*} 
-d\xi_p(v_i)&=& \left( -\frac{v_i\cdot m}{h^3f}(gh^2+\kappa_i f^2) \, y\, ,\, - \frac{\kappa_i(v_i\cdot m)}{h^3} \, m+gh^{-1}v_i +\frac{g(v_i\cdot m)\kappa_i}{h^3}\, N+h^{-1}\kappa_i v_i\right)\\
&=& \left( -\frac{v_i\cdot m}{h^3f}(gh^2+\kappa_i f^2) \, y\, ,\,gh^{-1}v_i +\frac{(v_i\cdot m)\kappa_i}{h^3}\, u+h^{-1}\kappa_i v_i\right)\\
\end{eqnarray*}

where $u=gN-m$. Since $u \cdot N=0$ then the vector $u\in T_mM_0$ and 

$$u=\sum_{j=1}^{n-k}(u\cdot v_j)\, v_j =-\sum_{j=1}^{n-k}(m\cdot v_j)\, v_j  .$$

Using this formula for $u$ in the expression above for $-d\xi_p(v_i)$ gives us that 

$$-d\xi_p(v_i)=\left( -\frac{v_i\cdot m}{h^3f}(gh^2+\kappa_i f^2) \, y\, ,\, \sum_{j=1}^{n-k}c_{ij}v_j\right)$$

where $c_{ij}=-\frac{1}{h^3}\kappa_i (v_i\cdot m)(v_j\cdot m) + (gh^{-1}+\kappa_ih^{-1})\delta_{ij}$. Recall that since $\xi$ is defined on $S^k\times M_0$ and not $M=\varphi(S^k\times M_0)$,  the expression above represents the directional derivative of the Gauss map defined on $M$ with respect to the vector $d\varphi(v_i)=\left(-\frac{v_i\cdot m}{f}\, y,v_i \right)$.

We also know that 

$$ -d\xi_p(v_i)=\sum_{j=i}^{n-k}b_{ij}d\varphi_p(v_j)= \left(-\sum_{j=1}^{n-k}b_{ij}\frac{v_i\cdot m}{f}\, y\, , \,\sum_{j=1}^{n-k }b_{ij}v_j \right) . $$

Since the vectors $\{v_1,\dots, v_{n-k}\}$ are linearly independent and $\sum_{j=1}^{n-k }b_{ij}v_j =\sum_{j=1}^{n-k }c_{ij}v_j $, then $c_{ij}=b_{ij}$. Therefore if $\{w_1,\dots,w_k\}$ is a basis for $T_yS^k$, then using the matrix representation of $-d\xi$ with respect to the basis $\{w_1,\dots,w_k,d\varphi(v_1), \dots d\varphi(v_{n-k})\}$, we obtain that the trace of $-d\xi$ is equal to 

$$nH = k\lambda_0+\sum_{i=1}^{n-k}c_{ii}=k \frac{g}{h}+(n-k)\frac{g}{h}+(n-k)\frac{H_0}{h}-\frac{1}{h^3}\sum_{i=1}^{n-k } \kappa_i(v_i\cdot m)^2=\frac{ng+(n-k)H_0}{h}-\frac{1}{h^3}\sum_{i=1}^{n-k } \kappa_i(v_i\cdot m)^2 .$$

\end{proof}


\begin{rem} When $M_0=S^{n-k}(r)$ is the $(n-k)$-dimensional sphere with radius $r$, then $M=S^k(\sqrt{1-r^2})\times S^{n-k}(r)$. If we take $N=\frac{m}{r}$, then $H_0=-\frac{1}{r}$, $g=r$, $f=h=\sqrt{1-r^2}$ and $v_i\cdot m=0$ for all $i$. Then we have that 

$$nH=\frac{-\frac{n-k}{r}+n r}{\sqrt{1-r^2}}= \frac{k-n+n r^2}{r\sqrt{1-r^2}} .$$

In this case the mean curvature of $M$ is well-known. The formula above agrees with the one provided in \cite{ABP}. 
\end{rem}


\section{Case $M_0$ is a surface of revolution}

Let us consider the case when $M_0$ is the surface obtained by revolving the curve $(f_1(u),f_2(u),0)$ around the $x$-axis. We will assume that this curve is parmetrized by arc length, that is, we will assume that 
$(f_1^\prime(u))^2+(f_2^\prime(u))^2=1$ for all $u$. We will also assume that $f_2(u)>0$. In this case 

\begin{eqnarray}\label{exrev}
\varphi(y,u,v)=\left(\sqrt{1-f_1^2(u)-f_2^2(u)}\, y\, ,\, f_1(u),f_2(u)\cos(v),f_2(u)\sin(v)\right) .
\end{eqnarray}

The immersion $M_0$ is given by $\left(\, f_1(u),f_2(u)\cos(v),f_2(u)\sin(v)\right)$ and its Gauss map is given by

$$N=\left( -f_2^\prime(u),f_1^\prime(u)\cos(v),f_1^\prime(u)\sin(v)\right)\, ,$$

the principal directions are given by the vectors

$$v_1=\left( f_1^\prime(u),f_2^\prime(u)\cos(v),f_2^\prime(u)\sin(v)\right)\quad\hbox{and}\quad v_2=(0,-\sin(v),\cos(v))$$

and the principal curvatures of $M_0$ are

\begin{eqnarray}\label{kappas}
\kappa_1= f_2^{\prime\prime}(u)f_1^\prime(u)-f_1^{\prime\prime}(u)f_2^\prime(u)
\quad\hbox{and}\quad \kappa_2=-\frac{f_1^\prime(u)}{f_2(u)} .
\end{eqnarray}

Moreover, we have that $(v_1\cdot m)=f_1(u)f_1^\prime(u)+f_2(u)f_2^\prime(u)=-f^\prime(u)f(u)$, $(v_2\cdot m)=0$ and

\begin{eqnarray}\label{fgh one variable}
f(u)=\sqrt{1-f_1(u)^2-f_2(u)^2},\quad g=f_2(u)f_1^\prime(u)-f_1(u)f_2^\prime(u),\quad h=\sqrt{1-g^2(u)}.
\end{eqnarray}

Since $M_0$ has dimension 2, $k=n-2$ and we have that 

$$nH=\frac{ng+\kappa_1+\kappa_2}{h}-\frac{1}{h^3}(\kappa_1(ff^\prime)^2) .$$

%
%

\subsection{Eigenvalues of the Laplace and stability operator.}

In this section we will assume that we have a $T$-periodic solution $\alpha(u)=(f_1(u),f_2(u))$ with $f_2(u)>0$, parametrized by arc-length that satisfies  $f_1^2(u)+f_2^2(u)<1$ for all $u$ that satisfies minimality equation 

$$ \frac{ng+\kappa_1+\kappa_2}{h}-\frac{1}{h^3}(\kappa_1(ff^\prime)^2)=0\, ,$$

where $f,g,h,\kappa_1$ and $\kappa_2$ are given in Equations \eqref{fgh one variable} and \eqref{kappas}. We have the following formula for the Laplace operator.

\begin{lem} \label{lap}For any function $\zeta:M\longrightarrow \mathbb{R}$, we have that 

$$\Delta \zeta=\frac{\Delta_{S^k}\zeta}{f^2}+
\frac{1}{1+(f^\prime)^2}
\left( \left( k\frac{f^\prime}{f}+\frac{f_2^\prime}{f_2}-\frac{f^\prime f^{\prime\prime}}{1+(f^\prime)^2}\right) \frac{\partial \zeta}{\partial u} +  \frac{\partial^2 \zeta}{\partial u^2}\right)+\frac{1}{f_2^2} \frac{\partial^2 \zeta}{\partial v^2} $$

where $\Delta_{S^k}$ is the Laplacian on the $k$ dimensional unit sphere.
\end{lem}

\begin{proof} Assume that $\eta:W\subset\mathbb{R}^k\longrightarrow S^k$ is a conformal parametrization. Let us use the indices $i,j$ with range from 1 to $k$. Denoting $\frac{\partial \eta}{\partial w_i}=\eta_i$ we have that $\eta_i\cdot \eta_j=\rho^2\delta_{ij}$. The first fundamental form of 

\begin{eqnarray}\label{parametrization}
\varphi(w,u,v)=\left(f(u)\eta(w) \, ,\, f_1(u)\, ,\, f_2(u)\cos(v)\, , \, f_2(u)\sin(v)\right)
\end{eqnarray} 

is diagonal, and the diagonal entries are given by 

$$\varphi_i\cdot \varphi_j=f^2\rho^2\delta_{ij},\quad \varphi_{k+1}\cdot \varphi_{k+1}=(f^\prime)^2+1\quad\hbox{and}\quad 
\varphi_{k+2}\cdot \varphi_{k+2}=f_2^2 . $$
We are denoting $\varphi_{k+1}=\frac{\partial \varphi}{\partial u}$ and   $\varphi_{k+2}=\frac{\partial \varphi}{\partial v}$. A direct computation shows that 

\begin{eqnarray*}
\varphi_i\cdot \nabla_{\varphi_i}\varphi_j= \varphi_i\cdot \varphi_{ij}=f^2\, (\eta_i\cdot \nabla_{\eta_i}\eta_j),\quad
\varphi_i\cdot \nabla_{\varphi_i}\varphi_{k+1}=\varphi_i\cdot \varphi_{i\, k+1}=ff^\prime\rho^2, \quad
\varphi_i\cdot \nabla_{\varphi_i}\varphi_{k+2}&=&\varphi_i\cdot \varphi_{i\, k+2}=0\\
\varphi_{k+1}\cdot \nabla_{\varphi_{k+1}}\varphi_{i}=0,\quad 
\varphi_{k+1}\cdot \nabla_{\varphi_{k+1}}\varphi_{k+1}=f^\prime f^{\prime\prime},\quad 
\varphi_{k+1}\cdot \nabla_{\varphi_{k+1}}\varphi_{k+2}=0\\
\varphi_{k+2}\cdot \nabla_{\varphi_{k+2}}\varphi_{i}=0,\quad 
\varphi_{k+2}\cdot \nabla_{\varphi_{k+2}}\varphi_{k+1}=f_2 f_2^{\prime},\quad \hbox{and}\quad
\varphi_{k+2}\cdot \nabla_{\varphi_{k+2}}\varphi_{k+2}=0 .
\end{eqnarray*}

We have that $\nabla\zeta=X_1+X_2+X_3$, where

$$X_1=\frac{1}{1+(f^\prime)^2}\frac{\partial \zeta}{\partial u} \varphi_{k+1},\quad
X_2=\frac{1}{f_2^2}\frac{\partial \zeta}{\partial v} \varphi_{k+2}\, \quad\hbox{and}\quad X_3=\sum_i \frac{1}{f^2\rho^2} \frac{\partial \zeta}{\partial w_i}\varphi_i . $$

Now, $\Delta \zeta=div(X_1)+div(X_2)+div(X_3)$. Computing the divergence of each vector field $X_i$ give us

\begin{eqnarray*}
div(X_1)&=& \frac{1}{1+(f^\prime)^2} \varphi_{k+1}\cdot \nabla_{\varphi_{k+1}} X_1 
+\frac{1}{f_2^2}\varphi_{k+2}\cdot \nabla_{\varphi_{k+2}} X_1 + \frac{1}{f^2\rho^2} \sum_j \varphi_j\cdot \nabla_{\varphi_j} X_1\\
&=&\left( \frac{1}{1+(f^\prime)^2} \zeta_u \right)_u+ \frac{f^\prime f^{\prime\prime}}{(1+(f^\prime)^2)^2}\zeta_u+
 \frac{1}{1+(f^\prime)^2}\frac{f_2^\prime}{f_2}\zeta_u+\sum_j \frac{1}{f^2} \frac{ff^\prime}{1+(f^\prime)^2}\zeta_u\\
 &=&  \frac{1}{1+(f^\prime)^2}\left(\left(k \frac{f^\prime}{f}+\frac{f_2^\prime}{f_2}-\frac{f^\prime f^{\prime\prime}}{1+(f^\prime)^2}\right)\, \zeta_u+ \zeta_{uu}\right)
\end{eqnarray*}

\begin{eqnarray*}
div(X_2)&=& \frac{1}{1+(f^\prime)^2} \varphi_{k+1}\cdot \nabla_{\varphi_{k+1}} X_2 
+\frac{1}{f_2^2}\varphi_{k+2}\cdot \nabla_{\varphi_{k+2}} X_2 + \frac{1}{f^2\rho^2} \sum_j \varphi_j\cdot \nabla_{\varphi_j} X_2\\
&=&0+\left(\frac{1}{f_2^2}\zeta_v\right)_v+0=\frac{1}{f_2^2}\zeta_{vv}
\end{eqnarray*}

and

\begin{eqnarray*}
div(X_3)&=& \frac{1}{1+(f^\prime)^2} \varphi_{k+1}\cdot \nabla_{\varphi_{k+1}} X_3 
+\frac{1}{f_2^2}\varphi_{k+2}\cdot \nabla_{\varphi_{k+2}} X_3 + \frac{1}{f^2\rho^2} \sum_j \varphi_j\cdot \nabla_{\varphi_j} X_3\\
&=&0+0+\sum_j \left(\sum_i \left( \frac{1}{f^2\rho^2}  \frac{\partial \zeta}{\partial w_i} \right)_{w_i}\delta_{ij}
+ \frac{1}{f^2\rho^4} \sum_i \frac{\partial \zeta}{\partial w_i}\left(\eta_j\cdot \nabla_{\eta_j}\eta_i\right) \right) =\frac{1}{f^2}\Delta_{S^k}\zeta .
\end{eqnarray*}

Combining the three expressions above the Lemma follows. \end{proof}

Since the stability operator for a  hypersurface in $S^{n+1}$ is given by $J(\zeta)=\Delta \zeta+|A|^2\zeta+n\zeta$, we need to find an expression for $|A|^2$. The following lemma will help us find this expression.

\begin{lem}
The principal curvatures of the hypersurface $M$ are:

$$\lambda_0=\frac{g}{h},\quad  \lambda_1=\frac{g+\kappa_1}{h}-\frac{\kappa_1(ff^\prime)^2}{h^3}\quad\hbox{and}\quad \lambda_2=\frac{g+\kappa_2}{h} . $$
\end{lem}

\begin{proof}
Let us consider the parametrization of $M$ defined in Equation (\ref{parametrization}). The argument in the proof of Lemma \ref{nH} gives us that the matrix of the  operator $-d\xi:T_mM\to T_mM$, with respect to the basis 
$\varphi_1,\dots \varphi_k,\varphi_{k+1},\varphi_{k+2}$, is the matrix $\begin{pmatrix}\lambda_0I_k& {\bf 0}\\{\bf 0}&C\end{pmatrix}$ where $I_k$ is the $k\times k$ identity matrix and the matrix $C$ is the two by two matrix with entries $c_{ij}=-\frac{1}{h^3}\kappa_i (v_i\cdot m)(v_j\cdot m) + (gh^{-1}+\kappa_ih^{-1})\delta_{ij}$. Since $v_1\cdot m=-ff^\prime$ and $v_2\cdot m=0$ we get that $C$ is a diagonal matrix and the lemma follows.
\end{proof}

In order to describe the spectrum of the Laplace operator on $M$ we will use the spectrum of the Laplace operator on sphere. The following proposition is well-known, see for example \cite{B}.

\begin{prop}\label{sLs}
The spectrum of the Laplace operator on the unit $k$-dimensional sphere is $\alpha_i=i(k+i-1)$, $i=0,\dots$ with multiplicities $m_0=1, \, m_1=k+1$ and 

$$m_i=\binom{k+i}{i}-\binom{k+i-2}{i-2}, \quad i= 2,\dots .$$
\end{prop}

Similar to the work done in \cite{P4} we have the following theorem.

\begin{thm}  \label{eigenvalues} Let $\varphi:S^k\times\mathbb{R}\times\mathbb{R}\longrightarrow M\subset S^{k+3} \subset \mathbb{R}^{k+4}$ be the immersion 

$$\varphi(y,u,v)=\left(\sqrt{1-f_1^2(u)-f_2^2(u)}\, y\, ,\, f_1(u),f_2(u)\cos(v),f_2(u)\sin(v)\right)$$

where $y\in S^k\subset \mathbb{R}^{k+1}$, $f_1$ and $f_2$ are $T$-periodic functions with $f_2(u)>0$, $f_1(u)^2+f_2(u)^2<1$, and 
 $f_1^\prime(u)^2+f_2^\prime(u)^2=1$
Let $\alpha_0=0$, $\alpha_1=k,\dots, \alpha_i=i(k+i-1)$ denote the eigenvalues of the Laplacian of the $k$-dimensional unit sphere $S^k$. The spectrum of the Laplace operator $\Delta$ on $M$ is given by the union $\cup_{i,j=0}^\infty\Gamma_{ij}$ where 

$$\Gamma_{ij}=\{\lambda_{ij}(1),\lambda_{ij}(2),\dots\}$$ 

is the ordered spectrum of the second order differential equation

$$L_{ij}(\eta)= \frac{1}{1+(f^\prime)^2}\left(\eta^{\prime\prime}+\left(k\frac{f^\prime}{f}+\frac{f_2^\prime}{f_2}-\frac{f^\prime f^{\prime\prime}}{1+(f^\prime)^2}\right)\eta^\prime\right)-\left(\frac{\alpha_i}{f^2}+\frac{j^2}{f_2^2}\right) \eta$$

defined on the set of $T$-periodic real functions.

Likewise, the spectrum of the stability operator $J=\Delta+k+2+|A|^2$ on $M$ is given by the union $\cup_{i,j=0}^\infty\bar{\Gamma}_{ij}$ where 

$$\bar{\Gamma}_{ij}=\{\bar{\lambda}_{ij}(1),\bar{\lambda}_{ij}(2),\dots\}$$ 

is the ordered spectrum of the second order differential equation

$$S_{ij}(\eta)= \frac{1}{1+(f^\prime)^2}\left(\eta^{\prime\prime}+\left(k\frac{f^\prime}{f}+\frac{f_2^\prime}{f_2}-\frac{f^\prime f^{\prime\prime}}{1+(f^\prime)^2}\right)\eta^\prime\right)-\left(\frac{\alpha_i}{f^2}+\frac{j^2}{f_2^2}-n-\left(k\lambda_0^2+\lambda_1^2+\lambda_2^2\right)\right) \eta .$$


\end{thm}

\begin{proof}
Let us assume that $\xi_i:S^k\to\mathbb{R}$ satisfies $\Delta_{S^k}\xi_i(y)+\alpha_i\xi_i(y)=0$;  
$g_j:\mathbb{R}\to\mathbb{R}$ is a $2\pi$-periodic function satisfying $g^{\prime\prime}(v)=-j^2g(v)$;  
and $\eta_l:\mathbb{R}\to\mathbb{R}$ is a $T$-periodic function satisfying $L_{ij}(\eta_l)+\lambda_{ij}(l)\eta_l=0$.  
Using Lemma \ref{lap}, we see that the function $\gamma_{ijl}:M\to\mathbb{R}$ defined by
\[
\gamma_{ijl}(\varphi(y,u,v))=\xi_i(y)g_j(v)\eta_l(u)
\]
satisfies $\Delta\gamma_{ijl}+\lambda_{ij}(l)\gamma_{ijl}=0$.  

To show that all eigenfunctions of the Laplacian are of the form $\gamma_{ijl}$, it is enough to prove that every function on $M$ can be expressed as a sum of such functions.  
Indeed, if $\gamma$ is a function on $M$, then for any $(y,u)$ we may expand it in Fourier series with respect to $v$:
\[
\gamma(y,u,v)=\sum_{j=0}^\infty c_j(y,u)g_j(v).
\]
Here the $c_j(y,u)$ are the Fourier coefficients.  
Next, for fixed $u$, each coefficient admits a spherical harmonic expansion in $y$:
\[
c_j(y,u)=\sum_{i=0}^\infty e_i(u)\xi_i(y),
\]
where the $\xi_i$ are eigenfunctions of $\Delta_{S^k}$ with eigenvalues $\alpha_i$.  
Finally, each $e_i(u)$ can be written as a sum of $T$-periodic eigenfunctions $\eta_l$ of the operator $L_{ij}$.  
Therefore every function $\gamma$ on $M$ can be expressed as a linear combination of the functions $\gamma_{ijl}$, which proves the claim for the spectrum of the Laplacian.  
The proof for the spectrum of the stability operator is analogous.
\end{proof}


\section{Case $M_0$ is a hypersurface of revolution}

In this section we will deduce the equation for the mean curvature of the hypersurface $M\subset S^{n+1}$ when  $M_0\subset \mathbb{R}^{l+2}$ is the hypersurface obtained by revolving the curve $(0,\dots,0,f_2(u),f_1(u))$ around the $x_{l+2}$-axis. Recall that in this case $n-k+1=l+2$. We will assume that 
$(f_1^\prime(u))^2+(f_2^\prime(u))^2=1$ and that $f_2(u)>0$  for all $u$. In this case 

\begin{eqnarray}\label{exrevwithl}
\varphi(y,u,z)=\left(\sqrt{1-f_1^2(u)-f_2^2(u)}\, y\, ,\, f_2(u)z,\, f_1(u)\right)
\end{eqnarray}

where, as in the general case $y\in S^k\subset \mathbb{R}^{k+1}$ and now $z\in S^l\subset \mathbb{R}^{l+1}$. The points in the hypersurface $M_0$ are of the form 

$$m=\left(\, f_2(u)\, z,\, f_1(u)\,\right)\, ,$$

and the Gauss map  of the hypersurface $M_0$ is given by 

$$N=\left( f_1^\prime(u)\, z\, ,-f_2^\prime(u)\right)\, .$$

If $u_1,\dots, u_l$ form a basis for $T_zS^l\subset \mathbb{R}^{l+1}$ then, the principal directions of $M_0$ are $v_1=(u_1,0),\dots,v_l=(u_l,0)$ and $v_{l+1}=\frac{\partial}{\partial u}=\left( f_2^\prime(u) \, z,\,  f_1^\prime(u)\, \right)$ and the principal curvatures of $M_0$ are

$$\kappa_1=\dots=\kappa_l=-\frac{f_1^\prime(u)}{f_2(u)}\quad\hbox{and}\quad  \kappa_{l+1}= f_2^{\prime\prime}(u)f_1^\prime(u)-f_1^{\prime\prime}(u)f_2^\prime(u) .$$

Moreover, if we define  $g$, $h$ and $f$ as in Lemma \ref{fgh}, then we have that  

$$(v_1\cdot m)=\dots=(v_l\cdot m)=0\quad \hbox{and}\quad (v_{l+1}\cdot m)=f_1(u)f_1^\prime(u)+f_2(u)f_2^\prime(u)=-f^\prime(u)f(u)$$

 and

\begin{eqnarray}\label{gH}
nH=\frac{ng+l \kappa_1+\kappa_{l+1}}{h}-\frac{1}{h^3}(\kappa_{l+1}(ff^\prime)^2) .
\end{eqnarray}

The principal curvatures for the immersion $M$ are given by 

\begin{eqnarray}\label{gprinccurv}
\lambda_1=\dots=\lambda_k=\frac{g}{h},\quad \lambda_{k+1}=\dots=\lambda_{k+l}=\frac{g+\kappa_{1}}{h},\quad \lambda_{k+l+1}=\frac{g+\kappa_{l+1}}{h}-\frac{\kappa_{l+1}(ff^\prime)^2}{h^3} .
\end{eqnarray}

All the computations above are very similar to the ones obtained when $M_0$ is a surface of revolution. We can also extend Lemma \ref{lap}. The proof is similar and will be omitted.

\begin{lem} \label{glap} Let $M$ be the immersion of $S^k\times S^l\times \mathbb{R}$ given by

$$\varphi(y,z,u)=\left(\sqrt{1-f_1^2(u)-f_2^2(u)}\, y\, ,\, f_2(u)z,\, f_1(u)\right)$$

where $y\in S^k$, $z\in S^l$ and $(f_1(u),f_2(u))$ defines an immersion of $S^1$, with $f_1$ and $f_2$ $T$-periodic functions and $f_2(u)>0$  and $f_1(u)^2+f_2(u)^2<1$ for all $u$.  For any function $\zeta:M\longrightarrow \mathbb{R}$ we have that 

$$\Delta \zeta=\frac{\Delta_{S^k}\zeta}{f^2}+
\frac{1}{1+(f^\prime)^2}
\left( \left( k\frac{f^\prime}{f}+l\frac{f_2^\prime}{f_2}-\frac{f^\prime f^{\prime\prime}}{1+(f^\prime)^2}\right) \frac{\partial \zeta}{\partial u} +  \frac{\partial^2 \zeta}{\partial u^2}\right)+ \frac{\Delta_{S^l}\zeta}{f_2^2}   $$

where $\Delta_{S^k}$ is the Laplacian on the $k$-dimensional unit sphere and $\Delta_{S^l}$ is the Laplacian on the $l$-dimensional unit sphere.
\end{lem}


\begin{thm}  \label{geigenvalues} Let $M$ be the immersion of $S^k\times S^l\times \mathbb{R}$ given by the immersion

$$\varphi(y,z,u)=\left(\sqrt{1-f_1^2(u)-f_2^2(u)}\, y\, ,\, f_2(u)z,\, f_1(u)\right)$$

where $y\in S^k$, $z\in S^l$ and $(f_1(u),f_2(u))$ defines an immersion of $S^1$, with $f_1$ and $f_2$ $T$-periodic functions and $f_2(u)>0$ and $f_1(u)^2+f_2(u)^2<1$  for all $u$.
Let $\alpha_0=0$, $\alpha_1=k,\dots, \alpha_i=i(k+i-1)$ be the eigenvalue of the Laplacian of the unit $k$-dimensional unit sphere $S^k$ and
 $\beta_0=0$, $\beta_1=l,\dots, \beta_i=i(l+i-1)$ be the eigenvalue of the Laplacian of the unit $l$-dimensional unit sphere $S^l$. The spectrum of the Laplace operator $\Delta$ on $M$ is given by the union $\cup_{i,j=0}^\infty\Gamma_{ij}$ where 

$$\Gamma_{ij}=\{\lambda_{ij}(1),\lambda_{ij}(2),\dots\}$$ 

is the ordered spectrum of the second order differential equation

$$L_{ij}(\eta)= \frac{1}{1+(f^\prime)^2}\left(\eta^{\prime\prime}+\left(k\frac{f^\prime}{f}+l\frac{f_2^\prime}{f_2}-\frac{f^\prime f^{\prime\prime}}{1+(f^\prime)^2}\right)\eta^\prime\right)-\left(\frac{\alpha_i}{f^2}+\frac{\beta_j}{f_2^2}\right) \eta$$

defined on the set of $T$-periodic real functions. 

Likewise, the spectrum of the stability operator $J=\Delta+k+l+1+|A|^2$ on $M$ is given by the union $\cup_{i,j=0}^\infty\bar{\Gamma}_{ij}$ where 

$$\bar{\Gamma}_{ij}=\{\lambda_{ij}(1),\lambda_{ij}(2),\dots\}$$ 

is the ordered spectrum of the second order differential equation

$$S_{ij}(\eta)= \frac{1}{1+(f^\prime)^2}\left(\eta^{\prime\prime}+\left(k\frac{f^\prime}{f}+l\frac{f_2^\prime}{f_2}-\frac{f^\prime f^{\prime\prime}}{1+(f^\prime)^2}\right)\eta^\prime\right)-\left(\frac{\alpha_i}{f^2}+\frac{\beta_j}{f_2^2}-n-\left(k\lambda_1^2+l\lambda_{k+1}^2+\lambda_{k+l+1}^2\right)\right) \eta . $$

Here, the functions $\lambda_i$ are defined in Equation \eqref{gprinccurv}

\end{thm}


\section{The ODE  and a numerical example}

The following theorem shows the ordinary differential equation that  $f_1$ and $f_2$ must satisfy  to generate immersed minimal hypersurfaces from $S^k\times S^l\times \mathbb{R}$ to $S^n$. Later on, we will study this ODE and provide some numerical embedded compact examples for several  values of $k$ and $l$.

\begin{thm}
Let  $\alpha(t)=(f_1(t),f_2(t))$ be a curve in the interior of the unit disk parametrized by arc-length, with $f_2(t)>0$ and let $\theta(t)$ be a smooth function such that $f_1^\prime(t)=\cos(\theta)$ and $f_2^\prime(t)=\sin(\theta)$. The immersion $\varphi:S^k\times S^l\times \mathbb{R}$ given in \ref{glap} satisfies $nH=0$ with $nH$ given in Equation (\ref{gH}), if and only if, 

\begin{eqnarray}\label{K}
\theta^\prime=K= \frac{\ \left(  (l-nf_2^2 ) \cos (\theta )+f_1 f_2 n \sin (\theta )\right) \left(1-(f_1 \sin (\theta )-f_2 \cos(\theta ))^2\right)}{f_2 \left(1-f_1^2-f_2^2\right)} .
\end{eqnarray}

\end{thm}

\begin{proof} A direct verification shows that the expression for $nH h^3$ reduces to 

$$\left(1-\left(f_2 f_1'-f_1 f_2'\right){}^2\right) \left(\frac{f_1' \left(f_2^2 n+f_2 f_2''-l\right)}{f_2}-f_2' \left(f_1 n+f_1''\right)\right)-\left(f_1 f_1'+f_2 f_2'\right){}^2 \left(f_1' f_2''-f_2' f_1''\right)$$

and if we make $f_1'=\cos(\theta)$, $f_2'=\sin(\theta)$,$f_1''=-\theta'\sin(\theta)$, $f_2''=\theta'\cos(\theta)$ then, the expression above transforms into

$$\left(1-\left(f_2 \cos (\theta )-f_1 \sin (\theta )\right){}^2\right) \left(-\frac{l \cos (\theta )}{f_2}-n f_1  \sin (\theta )+f_2 n \cos (\theta )+\theta '\right)-\theta ' \left(f_2 \sin (\theta )+f_1 \cos (\theta )\right){}^2 . $$

Setting the expression above equal to zero and solving for $\theta'$ completes the proof of the theorem.

\end{proof}

In order to find immersed examples we need to find periodic solutions of the ODE,

\begin{eqnarray}
\begin{cases} f_1'&=\cos(\theta)\\
f_2'&=\sin(\theta)\\\label{ODE}
\theta'&= \frac{\ \left(  (l-nf_2^2 ) \cos (\theta )+f_1 f_2 n \sin (\theta )\right) \left(1-(f_1 \sin (\theta )-f_2 \cos(\theta ))^2\right)}{f_2 \left(1-f_1^2-f_2^2\right)} .
\end{cases}
\end{eqnarray}

In \cite{Carl}, Carlotto and Schulz showed the existence of periodic solutions of this system  when $l=k$; this is when $n=2 k+1$.
 
  \begin{lem}\label{ode1}
 The solutions $f_1$ and $f_2$ of the system (\ref{ODE}) with initial conditions $\theta(0)=0,\, f_1(0)=0$ and $f_2(0)=a_0$ are periodic if for some positive number $T$, we have that $f_1(T/2)=0$, $f_2(T/2)=a_0$ and $\theta(T/2)=\pi$.
 \end{lem}

 \begin{proof} The lemma follows from the uniqueness of the solution of the system after  checking that the functions $\tilde{f}_1(t)=-f_1(-t)$, $\tilde{f}_2(t)=f_2(-t)$ and $\tilde{\theta}(t)=-\theta(-t)$ must agree with the functions  $f_1(t)$, $f_2(t)$ and $\theta(t)$  respectively, because they also satisfy the ODE with the same initial conditions as the solution $f_1(t),f_2(t),\theta(t)$.

 \end{proof}

\begin{ex} \label{ex1} We can numerically check that if we take $n=5$, $a=a_0=0.14971329$ and $T=2.0293246$, then $|f_1(T/2)|<10^{-8}$,  $|\theta(T/2)-\pi|<10^{-7}$. This example defines a minimal embedding of $S^3\times S^1\times S^1$ in $S^6$.

Figures \ref{tori} and \ref{circle} show the surface of revolution $M_0$ that generates the immersion $M$ in \ref{ex1}.
\begin{figure}[h]
\centerline
{\includegraphics[scale=0.52]{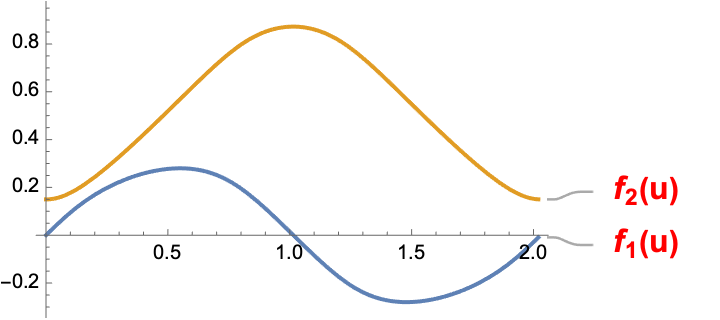} \includegraphics[scale=0.52]{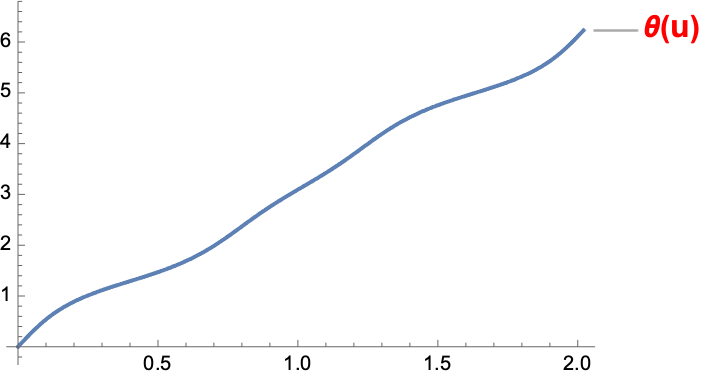}}
\caption{Solution of the system with $k=3$ and $a_0=0.14971331$. }
\label{tori}
\end{figure}

\begin{figure}[h]
\centerline
{\includegraphics[scale=0.62]{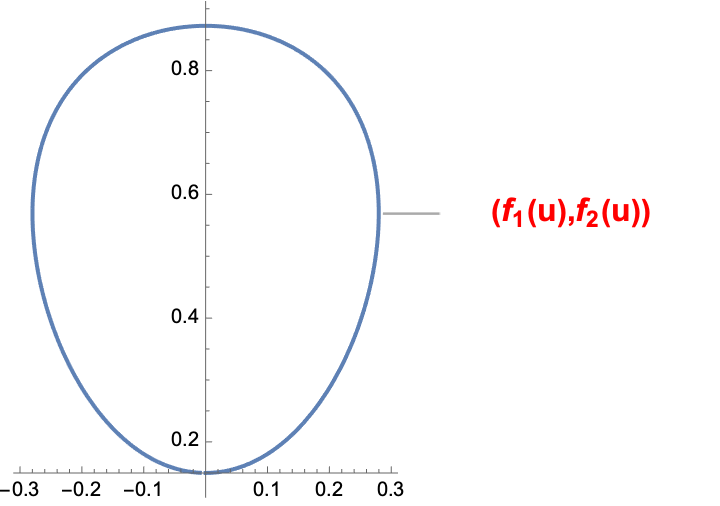} \includegraphics[scale=0.52]{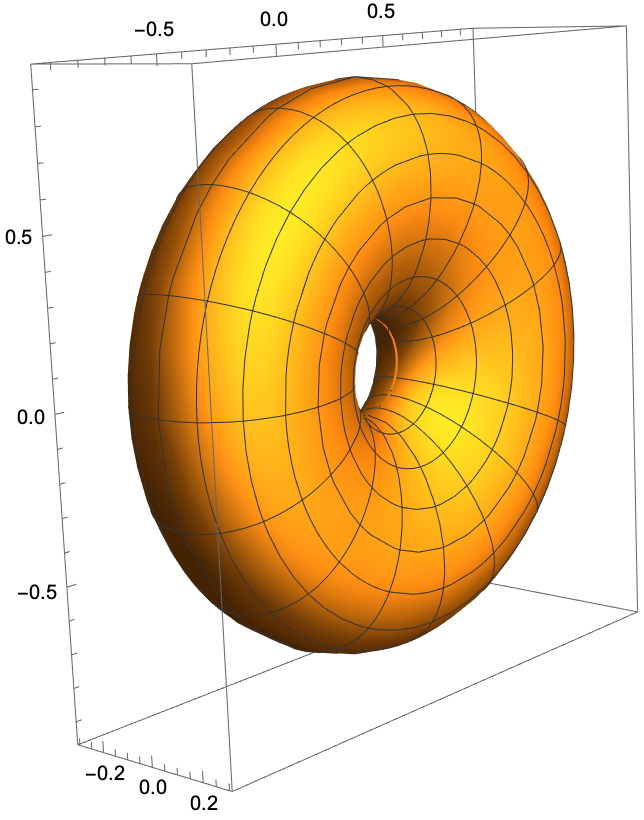}}
\caption{Graph of the surface $M_0$. The torus on the right is the rotation around the $x_1$-axis of the closed curve on the left. }
\label{circle}
\end{figure}
\end{ex}


\section{Numerical computation of the first eigenvalues of the Laplacian and  stability operator.}\label{numerical}

In this section we will compute the first eigenvalues of the Laplace and stability operator for the embedded example provided in Example \ref{ex1}. We will need the following lemma.

\begin{lem}\label{delta} Let us assume that $a(t)$, $b(t)$ and $c(t)$ are $T$-periodic functions with $a(t)>0$. If we define

$$L(u(t))=a(t) u^{\prime\prime}(t)+b(t)u^\prime(t)+c(t),$$

then a real number $\lambda$ satisfies $L(u)+\lambda u=0$ for a non-zero $T$-periodic solution $u$ if and only if

$$ \delta_0(\lambda) =1+z_1(T)z_2^\prime(T)-z_2(T)z_1^\prime(T)-\left(z_1(T)+z_2^\prime(T)\right) \, =\, 0\, ,$$

where  $z_1(t)$ satisfies $L(z_1)+\lambda z_1=0$ with $z_1(0)=1$ and $z_1^\prime(0)=0$ and $z_2(t)$ satisfies $L(z_2)+\lambda z_2=0$ with $z_2(0)=0$ and $z_2^\prime(0)=1$. 
\end{lem}

\begin{proof} We have that any non-zero solution $u$ of the differential equation $L(u)+\lambda u$ can be written as 

$$u(t)=c_1z_1(t)+c_2 z_2(t)$$

for some $(c_1,c_2)\ne(0,0)$. The solution $u(t)$ is periodic if and only if $u(T)=u(0)$ and $u^\prime(T)=u^\prime(0)$. The previous two equations are equivalent to the system 

$$\begin{cases}(z_1(T)-1) \, c_1+z_2(T) \, c_2=&0\\ z_1^\prime(T)\, c_1+(z_2^\prime(T)-1) \, c_2=&0\end{cases} $$

and we clearly have that the existence of a non-zero solution $(c_1,c_2)$ is equivalent to the condition: ``the determinant of the matrix of the 2 by 2 system vanishes''. That is, the existence of a non-zero solution $u(t)$ is equivalent to the condition $(z_1(T)-1)(z_2^\prime(T)-1)-z_2(T)z_1^\prime(T)=0$. Therefore the lemma follows.
\end{proof}

\begin{rem} When $b(t)$ is the zero function, we can check that the function $w=z_1z_2^\prime-z_2z_1^\prime$ satisfies that $w(0)=1$ and
$w^\prime(t)=0$. Therefore $z_1(T)z_2^\prime(T)-z_2(T)z_1^\prime(T)=1$ and the condition $\delta_0(\lambda)=0$ can be replaced with the condition $\delta=z_1(T)+z_2^\prime(T)=2$
\end{rem}

All numerical results in this section were obtained by solving the relevant second-order ODEs using standard numerical integration in Mathematica. The zeros of the discriminant functions were first identified by graphical inspection and then refined by root-finding routines, to an accuracy of $10^{-7}$.

\subsubsection{First Eigenvalues of the Laplacian of the immersion $M$ in Example \ref{ex1}.} 
In this subsection we will compute all the eigenvalues for the Laplace operator that are smaller than 12. We will show that they are $\lambda=0$ with multiplicity 1, $\lambda=5$ with multiplicity 7 , $\lambda=9.596\dots$ with multiplicity 2,  $\lambda=10.073\dots$ with multiplicity 4, $\lambda=10.6583\dots$ with multiplicity 1 and $\lambda=11.815\dots$ with multiplicity 8. 

Following Theorem \ref{eigenvalues} and Lemma \ref{delta} we need to find the zeroes of the functions $\delta_{ij}(\lambda)$ associated with the operators 

$$L_{ij}(\eta)= \frac{1}{1+(f^\prime)^2}\left(\eta^{\prime\prime}+\left(k\frac{f^\prime}{f}+\frac{f_2^\prime}{f_2}-\frac{f^\prime f^{\prime\prime}}{1+(f^\prime)^2}\right)\eta^\prime\right)-\left(\frac{\alpha_i}{f^2}+\frac{j^2}{f_2^2}\right) \eta . $$

In order to find these roots numerically we need to solve a second order equation for every $\lambda$. Since the coefficients of this second order equation depend on the functions $f_1$ and $f_2$ which were found numerically, then, for any $i,j$ we need to solve a system of equations that includes the solutions of the $f_1$ and $f_2$ along with the solution of the second order differential equations. To incorporate these two systems of equations into one, we need to write the coefficients of the second order equation in terms of $f_1,f_2$ and $\theta$. We accomplish this by using the following identities

$$f_1^\prime(u)=\cos(\theta(u)),\quad f_2^\prime(u)=\sin(\theta(u)),\quad f_1^{\prime\prime}(u)=-K\sin(\theta(u))\quad\hbox{and}\quad f_2^{\prime\prime}(u)=K\cos(\theta(u)), \quad \theta^\prime=K$$

where $K$ is given in Equation \eqref{K}. In this case the ``extended'' system can be written in the form

$$f_1^\prime=\cos(\theta),\,  f_2^\prime=\sin(\theta),\,  \theta^\prime=K, \, a(f_1,f_2,\theta) \, z^{\prime\prime}+b(f_1,f_2,\theta)\, z^\prime+c_{ij}(f_1,f_2,\theta,\lambda)\, z=0$$

with initial conditions $f_1(0)=0$, $f_2(0)=a_0=0.14971329$, $\theta(0)=0$ and $z(0)=1$ and $z^\prime(0)=0$ in order to find $z_1$ and; $z(0)=0$ and $z^\prime(0)=1$ in order to find $z_2$.

A direct computation shows that for the Laplace operator, we have that 

$$ a=\frac{f_1^2+f_2^2-1}{-f_1 f_2 \sin (2 \theta )+f_1^2 \sin ^2(\theta )+f_2^2 \cos ^2(\theta)-1}$$

$$ b=-5 f_1 \cos (\theta )-5 f_2 \sin (\theta )+\frac{\sin (\theta)}{f_2}\quad\hbox{and}\quad c_{ij}=\lambda-\frac{\alpha_k}{1-f_1^2-f_2^2}-\frac{j^2}{f_2^2} .$$ 


\begin{rem}[On the simplification of $b$]
An important observation is that if we fix $l=1$, then the expression for $b$ generalizes to every $n$. We have
\[
b=-n(f_1\cos\theta+f_2\sin\theta)+\frac{\sin\theta}{f_2}.
\]

To understand the simplification of the expression for $b$ above, it is convenient to change to the so-called TreadmillSled coordinates introduced by the author in the study of helicoidal surfaces (see for example \cite{PerdTS,P6}). We emphasize that, although the symbol $\xi$ is already used in this paper to denote the Gauss map, in this context we will temporarily denote the TreadmillSled coordinates. 

The new coordinates $(\xi_1,\xi_2)$ of the curve $\alpha=(f_1(t),f_2(t))$ are defined by
\[
\xi_1(t) = f_1(t)\cos(\theta(t)) + f_2(t)\sin(\theta(t)), 
\qquad 
\xi_2(t) = -f_1(t)\sin(\theta(t)) + f_2(t)\cos(\theta(t)),
\]
where $\theta$ is the angle determined by $\alpha'(t)=(\cos\theta,\sin\theta)$. With this change of variables we have  
\[
\theta' = \frac{\big( l \cos(\theta) - n f_2 \xi_2 \big)\,(1 - \xi_2^2)}{f_2 f^2}.
\]

A nice property of TS coordinates is
\[
\xi_1' = 1 + \xi_2 \theta', 
\qquad 
\xi_2' = -\xi_1 \theta', 
\qquad 
\xi_1^2 + \xi_2^2 = f_1^2 + f_2^2.
\]

Since 
\[
f = \sqrt{1 - f_1^2 - f_2^2} = \big(1 - \xi_1^2 - \xi_2^2 \big)^{1/2},
\]
we obtain
\[
f' = -\frac{1}{f}\,\xi_1,\qquad 
f'' = \frac{1}{f^3}\Big(f^2 \theta' - (1 - \xi_2^2)\Big),\qquad 
1 + f'^2 = \frac{f^2}{1 - \xi_2^2}.
\]

With these identities one can isolate the part proportional to $n$ and verify directly that it coincides with the expected term $-n(f_1\cos\theta+f_2\sin\theta)$. The remaining terms, independent of $n$, simplify to $\sin\theta/f_2$. To obtain these simplifications it is necessary to return to the original variables $(f_1,f_2)$, substitute back their relations with $\theta$, and simplify once more. This explains why the apparently complicated coefficient $b$ collapses to the very short expression displayed above.
\end{rem}

Let us study the operator $L_{00}$. Figure \ref{delta00Lap} is made up of 480 points. For the first point we took $\lambda=0$ and then we solved for $z_1$ and $z_2$ and then we computed  $\delta_{00}(0)$. For the second point we took $\lambda=0.025$ and we computed $\delta_{00}(0.025)$. We keep increasing the values of $\lambda$ by $0.025$ until we reached our last point $\lambda=11.975$ with a value for $\delta_{00}$ equal to $1.1755\dots$

\begin{figure}[h]
\centerline
{ \includegraphics[scale=0.52]{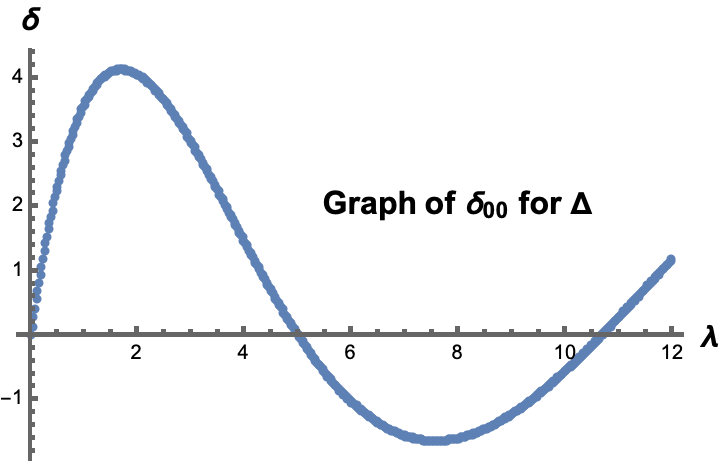}}
\caption{Graph of the function $\delta_{00}$ associated with the operator $L_{00}$ }
\label{delta00Lap}
\end{figure}

The values $\lambda=0$ and  $\lambda=5$ were expected. For $\lambda=0$ an eigenfunction is the constant function $1$ and for $\lambda=5$ an eigenfunction is the function $f_1$. The next zero for the function $\delta_{00}$ can be computed using the intermediate value theorem and it turns out to be $\lambda=10.658388\dots$. The first eigenfunctions of $L_{00}$ are shown in Figure \ref{EigenLambda00}. These functions are not only eigenfunctions for the operator $L_{00}$ but also, when viewed as functions on $M$, they are eigenfunctions of the Laplacian on $M$.

\begin{figure}[h]
\centerline
{ \includegraphics[scale=0.52]{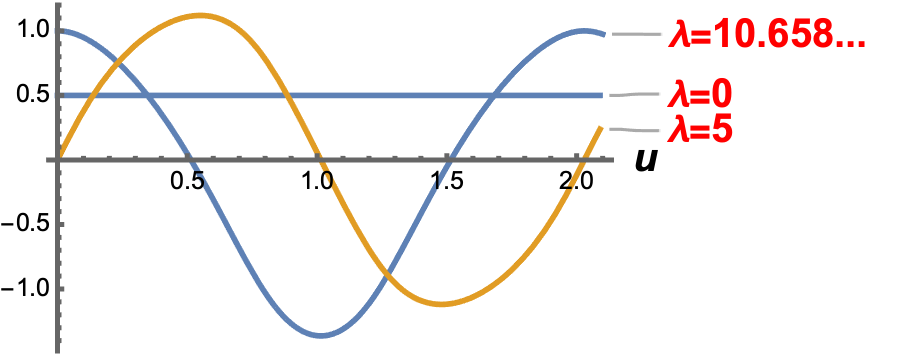}}
\caption{Graph of the first three eigenfunctions of $L_{00}$ }
\label{EigenLambda00}
\end{figure}

In order to study the operator $L_{10}$ we notice that  $\delta_{10}(0)$ is around $ -273.76$ and moreover, we notice that $\delta_{10}$ remains negative for values of $\lambda<5$. Therefore we decided to graph the function $\delta_{10}$ starting at $\lambda=4$.  This time the graph is made up of 350 points. Our first point corresponds to $\lambda=4$, the second one with $\lambda=4.025$ and so on. Similar considerations were made for the functions $\delta_{01}$ and $\delta_{11}$.  Figure \ref{deltasLap} shows these three graphs.

\begin{figure}[h]
\centerline
{ \includegraphics[scale=0.62]{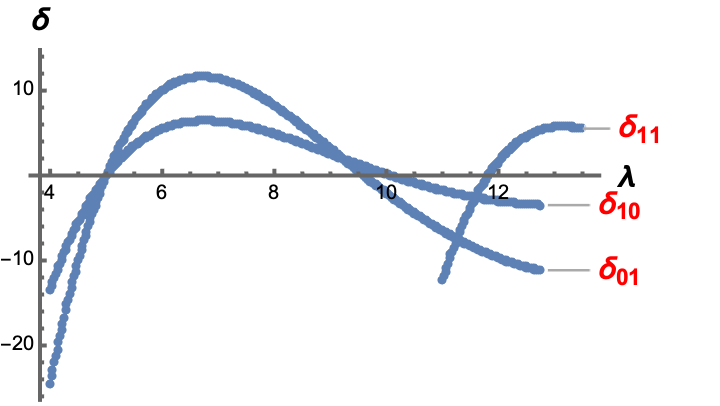}}
\caption{Graph of the functions $\delta_{10}$ associated with the operator $L_{10}$, $\delta_{01}$ associated with the operator $L_{01}$ and $\delta_{11}$ associated with the operator $L_{11}$. }
\label{deltasLap}
\end{figure}

The eigenfunction for $\lambda=5$ of the operator $L_{10}$ was expected and we can directly check that the positive function $f=\sqrt{1-f_1^2-f_2^2}$ is an eigenfunction of $L_{10}$ and the following four functions on $M$, $fy_1,\dots , fy_4$, where $y=(y_1,y_2,y_3,y_4)$ are the coordinates of $S^3\subset \mathbb{R}^4$, are four linearly independent eigenfunctions for $\Delta$ associated with $\lambda=5$. The next zero of the function $\delta_{10}$ turns out to be $\lambda=10.073635149\dots$. For this eigenvalue, an eigenfunction of the operator $L_{10}$ is shown in Figure \ref{EigenLambda10}. In this case the product of this eigenfunction with the functions $y_i$ provide linearly independent eigenfunctions for  the Laplacian on $M$.

\begin{figure}[h]
\centerline
{ \includegraphics[scale=0.52]{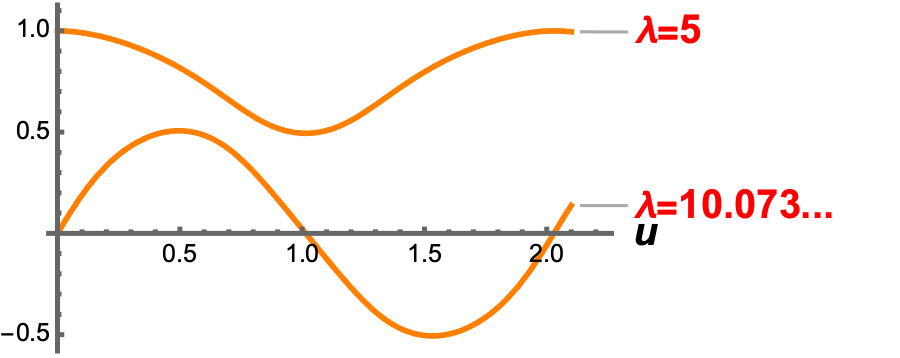}}
\caption{Graph of the first two eigenfunctions of $L_{10}$  }
\label{EigenLambda10}
\end{figure}

Let us review now the operator $L_{01}$. The first eigenvalue is $5$. A direct computation shows that $f_2$ is an eigenfunction associated with $\lambda=5$ and, we have that $f_2(u)\sin(v)$ and $f_2(u)\cos(v)$ are eigenfunctions of the Laplacian on $M$. The next zero for $\delta_{01}$ is $9.596...$ and an eigenfunction associated with $\lambda=9.5961595...$ is shown in Figure \ref{EigenLambda01}. If we call this eigenfunction $f_{9596}$ then, we have that the functions $f_{9596}(u)\cos(v)$ and $f_{9596}(u)\sin(v)$ are two linearly independent eigenfunctions of the Laplacian on $M$.

\begin{figure}[h]
\centerline
{ \includegraphics[scale=0.52]{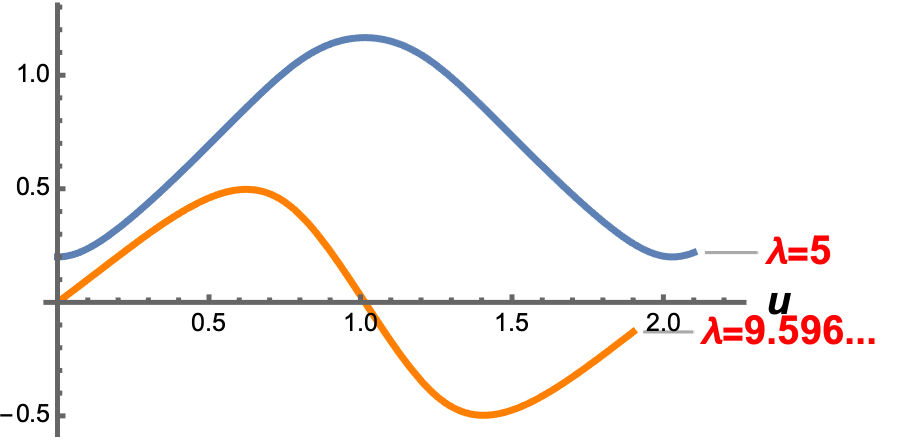}}
\caption{Graph of the first two eigenfunctions of $L_{01}$  }
\label{EigenLambda01}
\end{figure}

Let us review now the operator $L_{11}$. The first zero for $\delta_{11}$ is $11.815175...$ and an eigenfunction associated with $\lambda=11.815175...$ is shown in Figure \ref{EigenLambda11}. If we call this eigenfunction $f_{1181}$ then, we have that the functions $f_{1181}(u)\cos(v)y_i$ and $f_{1181}(u)\sin(v)y_i$ are eight linearly independent eigenfunctions of the Laplacian on $M$.

\begin{figure}[h]
\centerline
{ \includegraphics[scale=0.52]{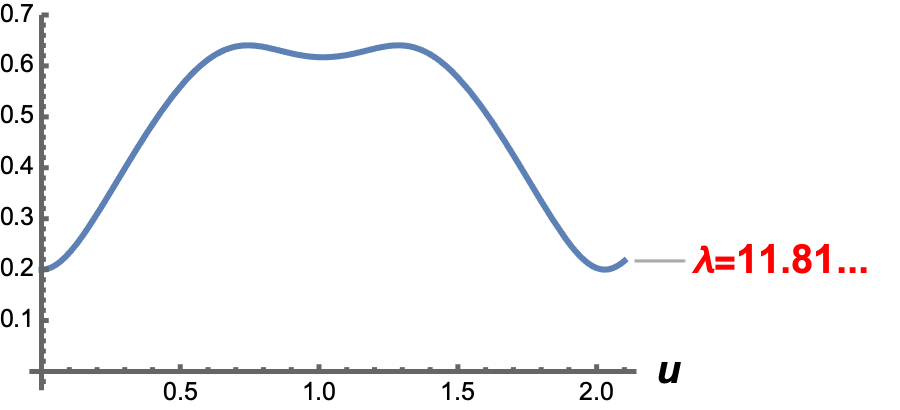}}
\caption{Graph of the first  eigenfunctions of $L_{11}$  }
\label{EigenLambda11}
\end{figure}

%
%
%


\subsubsection{Non-positive eigenvalues of the stability Operator of the immersion $M$ in Example \ref{ex1}.} 

In this section we will compute all non-positive eigenvalues of the stability operator $J(\eta)=\Delta\eta+(k+2)\eta+|A|^2\eta$. We will show that they are: $\lambda=-32.232\dots$ with multiplicity 1, $\lambda=-29.0007\dots$ with multiplicity 4, $\lambda=-23.630\dots$ with multiplicity 9, $\lambda=-16.133\dots$ with multiplicity 16, $\lambda=-14.662\dots$ with multiplicity 1, $\lambda=-13.476\dots$ with multiplicity 2, $\lambda=-8.255\dots$ with multiplicity 2,  $\lambda=-6.516\dots$ with multiplicity 25, $\lambda=-5$ with multiplicity 15, $\lambda=-0.4047\dots$ with multiplicity 2 and $\lambda=0$ with multiplicity 14. Therefore, the stability index, which is the number of negative eigenvalues counted with multiplicity is 77 and the nullity, which is the multiplicity of the eigenvalue $\lambda=0$ is 14.

Following Theorem \ref{eigenvalues} and Lemma \ref{delta} we need to find the zeroes of the functions $\bar{\delta}_{ij}(\lambda)$ associated with the operators 

$$S_{ij}(\eta)= \frac{1}{1+(f^\prime)^2}\left(\eta^{\prime\prime}+\left(3\frac{f^\prime}{f}+\frac{f_2^\prime}{f_2}-\frac{f^\prime f^{\prime\prime}}{1+(f^\prime)^2}\right)\eta^\prime\right)-\left(\frac{\alpha_i}{f^2}+\frac{j^2}{f_2^2}-n-\left(3\lambda_0^2+\lambda_1^2+\lambda_2^2\right)\right) \eta .$$

In order to find these roots numerically we need to solve a second order equation for some large amount of values for  $\lambda$ in a given interval. As in the case for the Laplacian, since the coefficients of this second order equation depend on the functions $f_1$ and $f_2$ which were found numerically, then, for any $i,j$ we need to solve a system that includes the solution of the $f_1$ and $f_2$ along with the solution of the second order equations. To incorporate these differential equations into one single system,  we need to write the coefficients of the second order equation in terms of $f_1,f_2$ and $\theta$. 
This time the extended system can be written as 

$$f_1^\prime=\cos(\theta),\,  f_2^\prime=\sin(\theta),\,  \theta^\prime=K, \, a(f_1,f_2,\theta) \, z^{\prime\prime}+b(f_1,f_2,\theta)\, z^\prime+\bar{c}_{ij}(f_1,f_2,\theta,\lambda)\, z=0$$

with initial conditions $f_1(0)=0$, $f_2(0)=a_0=0.149713296$, $\theta(0)=0$ and $z(0)=1$ and $z^\prime(0)=0$ in order to find $z_1$ and, $z(0)=0$ and $z^\prime(0)=1$ in order to find $z_2$. The functions $a$ and $b$ are those used to study the operators $\Gamma_{ij}$ and $\bar{c}_{ij}=c_{ij}+5+|A|^2$ where  $|A|^2$ is equal to 

$$ \frac{-20 f_1^2 f_2^2 \sin ^2(\theta )+5 f_1 f_2 \left(4 f_2^2-1\right) \sin (2 \theta )-2 \left(10 f_2^4-5 f_2^2+1\right) \cos ^2(\theta )}{f_2^2 (-f_1 \sin (\theta )+f_2 \cos (\theta )-1) (-f_1 \sin (\theta )+f_2 \cos (\theta )+1)} .$$

The graphs of the function $\bar{\delta}_{00}$,  $\bar{\delta}_{10}$, $\bar{\delta}_{01}$, $\bar{\delta}_{11}$, $\bar{\delta}_{20}$,  $\bar{\delta}_{30}$, $\bar{\delta}_{40}$, $\bar{\delta}_{02}$  and $\bar{\delta}_{03}$ are show in Figures \ref{deltabar1}, \ref{deltabar2} and \ref{deltabar3}. Using the intermediate value theorem we can check that all the zeroes of the functions $\bar{\delta}_{04}$, $\bar{\delta}_{50}$, $\bar{\delta}_{12}$ and $\bar{\delta}_{21}$ are positive. Using that $S_{ij}\leq S_{i'j'}$ whenever $i\leq i'$ and $j\leq j'$, and applying the Rayleigh principle, it follows that the first eigenvalue of $S_{i'j'}$ is always greater than or equal to that of $S_{ij}$. Hence it suffices to analyze only the finite list of discriminant graphs displayed above.

\begin{figure}[h]
\centerline
{ \includegraphics[scale=0.33]{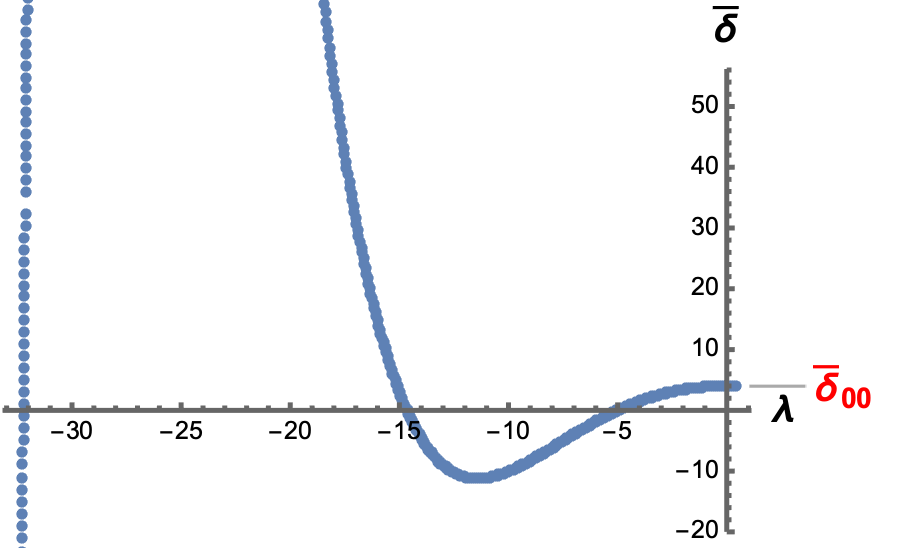}\hskip.1cm \includegraphics[scale=0.33]{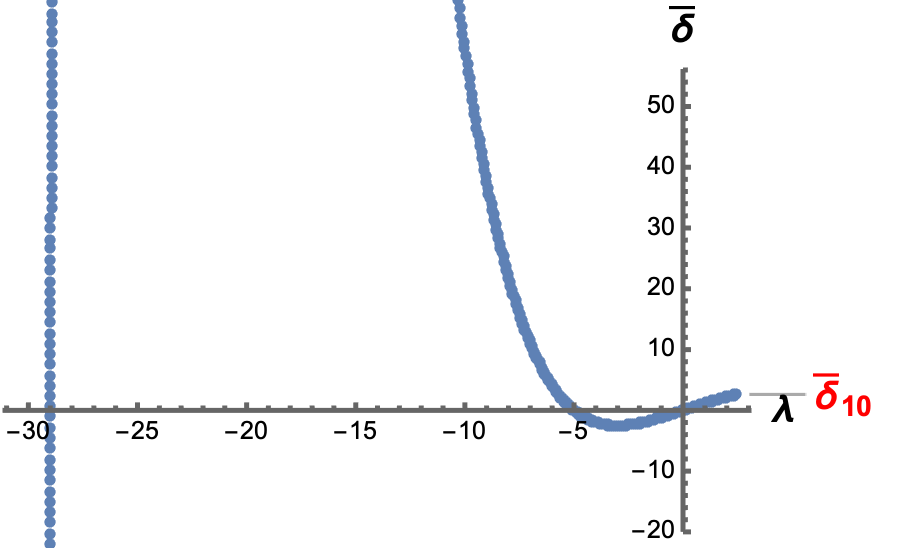}\hskip.1cm \includegraphics[scale=0.33]{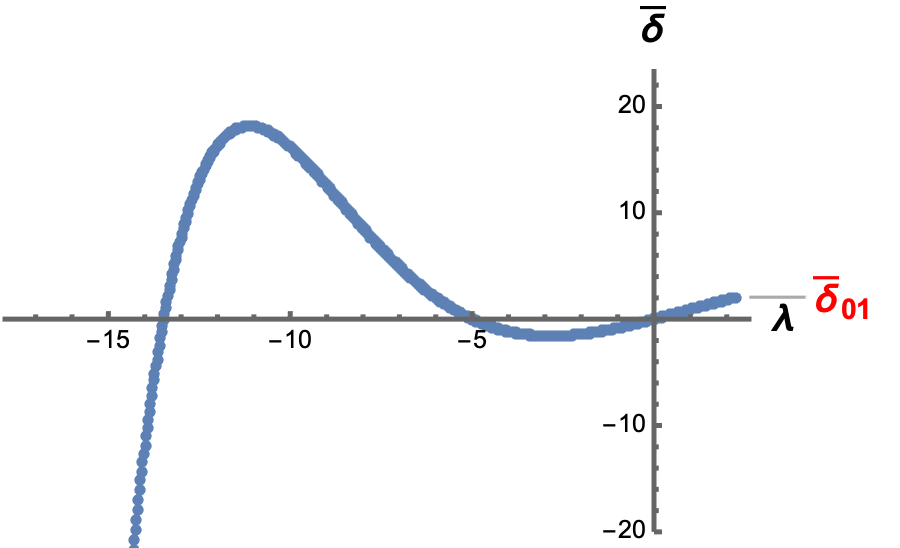}}
\caption{Graphs of the function $\bar{\delta}_{00}$, $\bar{\delta}_{10}$, $\bar{\delta}_{01}$,  associated with the operators $S_{00}$, $S_{10}$,  $S_{01}$ respectively.}
\label{deltabar1}
\end{figure}

\begin{figure}[h]
\centerline
{ \includegraphics[scale=0.33]{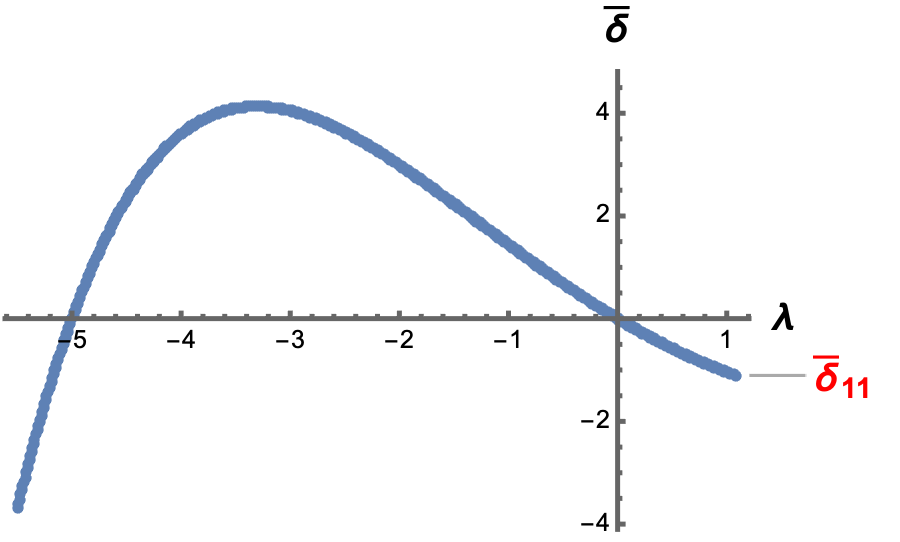}\hskip.1cm \includegraphics[scale=0.33]{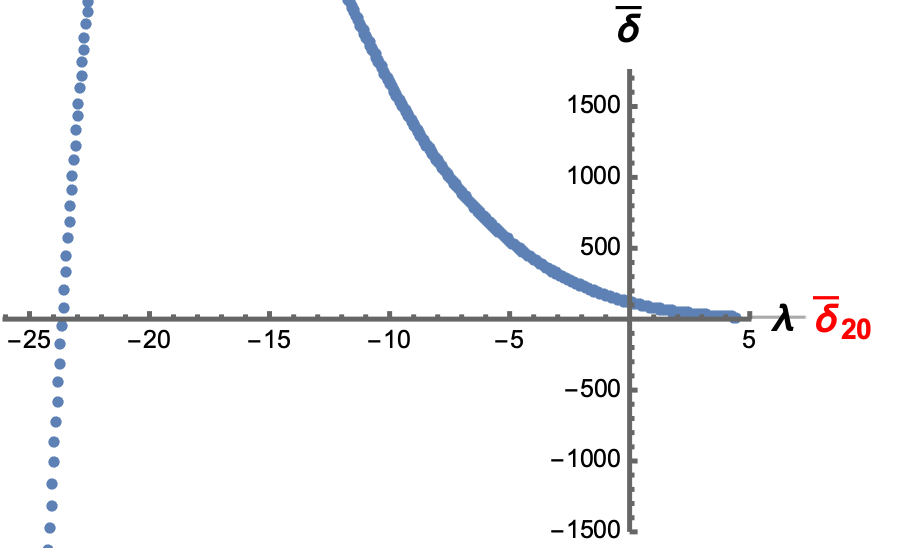}\hskip.1cm \includegraphics[scale=0.33]{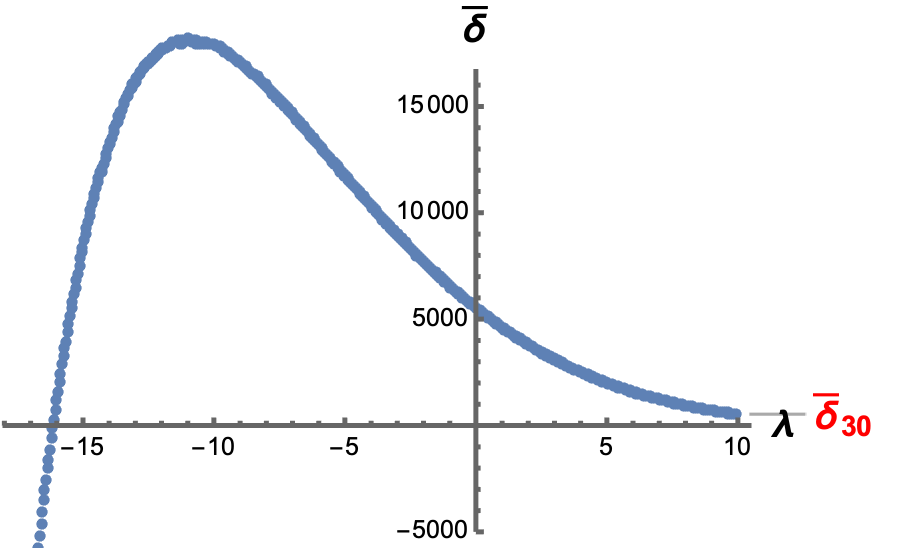}}
\caption{Graphs of the function $\bar{\delta}_{11}$, $\bar{\delta}_{20}$,  $\bar{\delta}_{30}$ associated with the operators $S_{11}$, $S_{20}$,  $S_{30}$ respectively.}
\label{deltabar2}
\end{figure}

\begin{figure}[h]
\centerline
{ \includegraphics[scale=0.33]{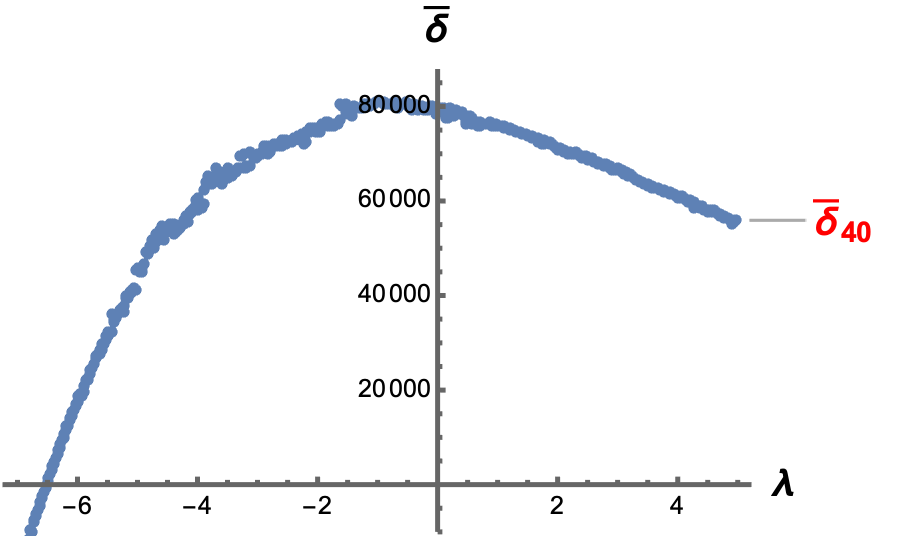}\hskip.1cm \includegraphics[scale=0.33]{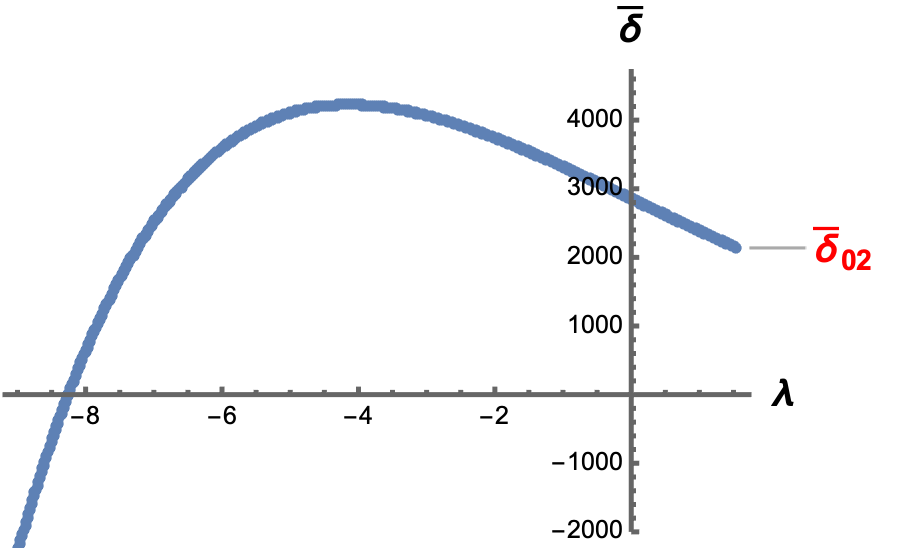}\hskip.1cm \includegraphics[scale=0.33]{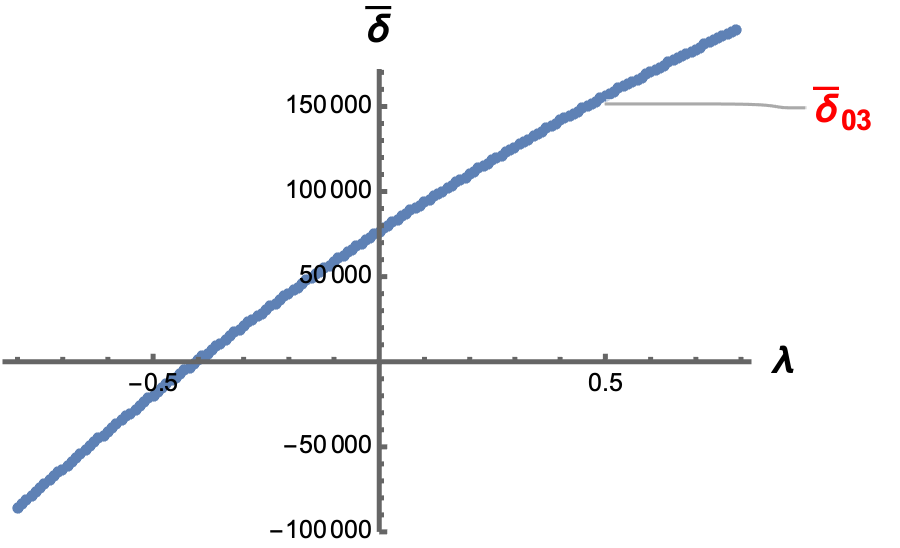}}
\caption{Graphs of the function $\bar{\delta}_{40}$, $\bar{\delta}_{02}$, $\bar{\delta}_{03}$ associated with the operators $S_{40}$, $S_{02}$,  $S_{03}$ respectively.}
\label{deltabar3}
\end{figure}

Before continuing counting the zeroes for the stability operator, let us denote by $\eta_{1,1}(y), \dots \eta_{1,4}(y)$ four linearly independent eigenfunctions of the Laplacian on $S^3$ associated with $\lambda=3$, $\eta_{2,1}(y), \dots \eta_{2,9}(y)$ nine linearly independent eigenfunctions of the Laplacian on $S^3$ associated with $\lambda=8$, $\eta_{3,1}(y), \dots \eta_{3,16}(y)$ sixteen linearly independent eigenfunctions of the Laplacian on $S^3$ associated with $\lambda=15$, $\eta_{4,1}(y), \dots \eta_{4,25}(y)$ twenty five  linearly independent eigenfunctions of the Laplacian on $S^3$ associated with $\lambda=24$.    

 As in the previous case, we can compute the zeros of the functions $\bar{\delta}_{ij}$ using the intermediate value theorem. 
 
\begin{rem}
For the hypersurface $M \subset S^6$ considered above, the Jacobi operator has the following eigenvalues (approximate) with the indicated multiplicities:
\[
\begin{array}{c|c}
\text{Eigenvalue} & \text{Multiplicity} \\
\hline
-32.232\dots & 1 \\
-29.0007\dots & 4 \\
-23.6309\dots & 9 \\
-16.133\dots & 16 \\
-14.662\dots & 1 \\
-13.476\dots & 2 \\
-8.255\dots & 2 \\
-6.516\dots & 25 \\
-5 & 15 \\
-0.4047\dots & 2 \\
0 & 14
\end{array}
\]
\end{rem}

\medskip
\noindent\textit{Eigenvalues of $S_{00}$.}  
The nonpositive eigenvalues are $\lambda=-32.232\dots, -14.662\dots, -5$.  
If we label the eigenfunctions of these eigenvalues as $f_{00,n323}, f_{00,n146}, f_{00,n5}$, then these functions, when viewed on $M$, are eigenfunctions of the stability operator.  
Hence $\lambda=-32.232\dots$ and $\lambda=-14.662\dots$ each have multiplicity one, and we obtain one function in the eigenspace of $\lambda=-5$.

\medskip
\noindent\textit{Eigenvalues of $S_{10}$.}  
The nonpositive eigenvalues are $\lambda=-29.0007\dots, -5, 0$.  
If we denote the corresponding eigenfunctions by $f_{10,n29}, f_{10,n5}, f_{10,0}$, then the functions $f_{10,n29}(u)\eta_{1i}(y)$, $f_{10,n5}(u)\eta_{1i}(y)$, and $f_{10,0}(u)\eta_{1i}(y)$ are eigenfunctions of the stability operator.  
Thus $\lambda=-29.0007\dots$ has multiplicity $4$, while $\lambda=-5$ and $\lambda=0$ each give rise to $4$ eigenfunctions.

\medskip
\noindent\textit{Eigenvalues of $S_{01}$.}  
The nonpositive eigenvalues are $\lambda=-13.476\dots, -5, 0$.  
If we denote the corresponding eigenfunctions by $f_{01,n13}, f_{01,n5}, f_{01,0}$, then each multiplied by $\cos(v)$ or $\sin(v)$ gives an eigenfunction of the stability operator.  
Hence $\lambda=-13.476\dots$ has multiplicity $2$, while $\lambda=-5$ and $\lambda=0$ each give rise to two eigenfunctions.

\medskip
\noindent\textit{Eigenvalues of $S_{11}$.}  
The nonpositive eigenvalues are $\lambda=-5, 0$.  
If we label the eigenfunctions as $f_{11,n5}, f_{11,0}$, then multiplying by $\cos(v)\eta_{1i}$ or $\sin(v)\eta_{1i}$ produces eigenfunctions of the stability operator.  
Thus $\lambda=-5$ and $\lambda=0$ each give rise to $8$ eigenfunctions.

\medskip
\noindent\textit{Eigenvalues of $S_{20}$, $S_{30}$ and $S_{40}$.}  
The only nonpositive eigenvalues are $\lambda=-23.6309\dots, -16.133\dots, -6.516\dots$, corresponding to $S_{20}$, $S_{30}$ and $S_{40}$ respectively. Labeling the eigenfunctions as $f_{20,n23}, f_{30,n16}, f_{40,n6}$, and multiplying by $\eta_{2i}, \eta_{3i}, \eta_{4i}$, we obtain eigenfunctions of the stability operator.  
The corresponding multiplicities are $9, 16, 25$.

\medskip
\noindent\textit{Eigenvalues of $S_{02}$ and $S_{03}$.}  
The only nonpositive eigenvalues are $\lambda=-8.255\dots$ and $\lambda=-0.4047\dots$, respectively.  
Labeling the eigenfunctions as $f_{02,n8}$ and $f_{03}$, each multiplied by $\cos(2v), \sin(2v)$ (resp. $\cos(3v), \sin(3v)$) gives eigenfunctions of the stability operator.  
Thus each has multiplicity $2$.

In total, the multiplicity of the eigenvalue $\lambda=-5$ is $1+2+4+8=15$, and the multiplicity of the eigenvalue $\lambda=0$ is $4+2+8=14$. The multiplicity for $\lambda=-5$ was unexpected: the expected value was $7$, since the coordinate functions of the Gauss map form a $7$-dimensional space (it is relatively easy to prove that these $7$ functions are linearly dependent only for a totally umbilical sphere $S^5\subset S^6$). On the other hand, the multiplicity for $\lambda=0$ was expected, since $14$ is the dimension of the space of Jacobi fields coming from isometries of the ambient space. These reduce to functions of the form $\varphi B\xi$, where $B$ is a skew-symmetric $7\times 7$ matrix, $\varphi$ is the immersion, and $\xi$ is the Gauss map.

Example \ref{n4to50} shows that there exist (at least numerically) embedded examples of minimal hypersurfaces of $S^{n+1}$  for values of $n$ between 4 and 50.

We can also find embedded numerical examples with $l>1$. 

\begin{ex} \label{ex2} Let us assume that $n=5$, $k=2$ and $l=2$  and let us consider the differential equation induced by the equation $H=0$, where $H$ is given by Equation (\ref{gH}). If we take $f_1^\prime=\cos(\theta)$ and $f_2^\prime=\sin(\theta)$ then the solution coming from the initial conditions $f_2(0)=a_0=0.3309805$, $f_1(0)=0$ and $\theta(0)=0$ satisfies that if  $T=1.8733685$, then $|f_1(T/2)|<5*10^{-7}$  and  $|\theta(T/2)-\pi|<10^{-7}$. Therefore, we (numerically) have a minimal immersion of $S^2\times S^2\times S^1$ in $S^6$ which is embedded. A non-numerical proof of the existence of this example is given by Carlotto and Schulz in \cite{Carl}; they show the existence of minimal embedded hypersurfaces from $S^{k-1}\times S^{k-1}\times S^1$ into $S^{2k}$. In  the recent paper \cite{Perd2025}, the author shows that the stability index of the Carlotto-Schulz minimal hypersurfaces is at least $k^2+4k+3$.  For the example considered here, the bound gives that the stability index is at least $24$, while numerically this index is shown to be $45$. Moreover, the nullity in this case is $15$.

Figure \ref{gtori} shows the graphs of the functions $f_1$, $f_2$ and $\theta$.
\begin{figure}[h]
\centerline
{\includegraphics[scale=0.52]{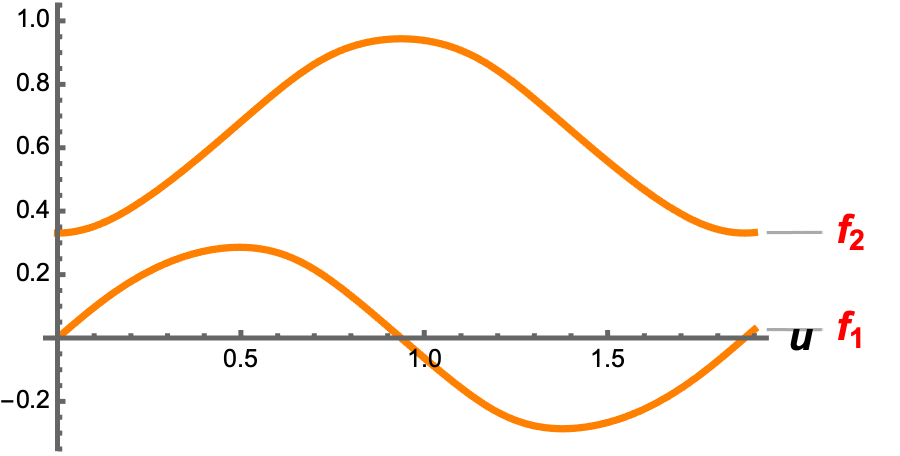} \hskip.3cm  \includegraphics[scale=0.52]{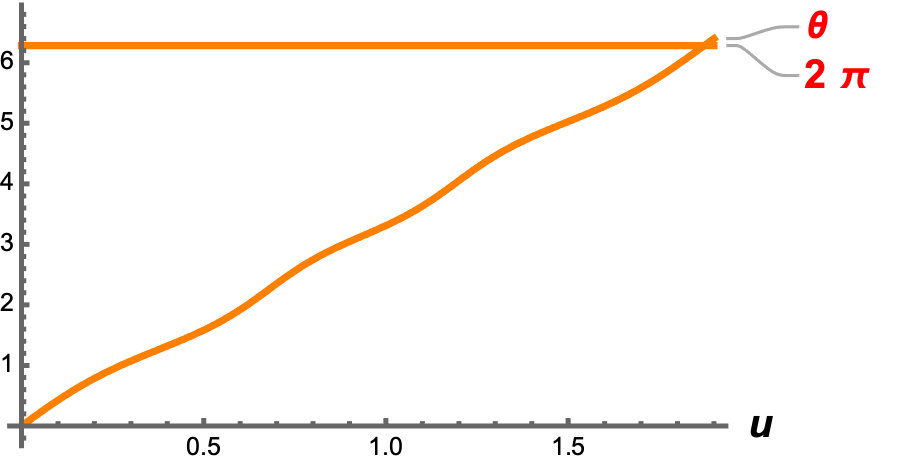}}
\caption{Graphs of the periodic solutions $f_1$, $f_2$ and the function $\theta$, when  $k=l=2$ and $a_0=0.3309805$. }
\label{gtori}
\end{figure}

\end{ex}

\begin{ex} \label{n4to50}
Let us consider immersions of the form described in Equation (\ref{exrev}) with $n$ given by the first column of the table below. Assume that  $f_1$, $f_2$ and $\theta$  satisfy the ODE defined in Lemma \ref{ode1}, with initial conditions $\theta(0)=0$, $f_1(0)=0$, and  $f_2(0)=a_0$ with $a_0$ given by the second column of the table below. We have that the values of  $T$ given by the third column of the table below satisfy $|f_1(T/2)|<10^{-6}$ and  $|\theta(T/2)-\pi|<10^{-6}$.  Therefore this list of numbers produces numerical examples of  embedded minimal hypersurfaces $M\subset S^{n+1}$ for $n=4,5\dots,50$.
 
\begin{table}[h!]
\centering\small
\setlength{\tabcolsep}{5pt}
\renewcommand{\arraystretch}{1.05}
\begin{tabular*}{\textwidth}{@{\extracolsep{\fill}} rcc rcc rcc @{}}
\hline
$n$ & $a_0$ & $T$ & $n$ & $a_0$ & $T$ & $n$ & $a_0$ & $T$ \\
\hline
4  & 0.16854  & 2.17363 & 5  & 0.149713 & 2.02932 & 6  & 0.135385 & 1.90413 \\
7  & 0.124316 & 1.79709 & 8  & 0.115504 & 1.70510 & 9  & 0.108296 & 1.62530 \\
10 & 0.102268 & 1.55538 & 11 & 0.097135 & 1.49355 & 12 & 0.0926974 & 1.43840 \\
13 & 0.0888125& 1.38884 & 14 & 0.0853753 & 1.34401 & 15 & 0.0823064 & 1.30320 \\
16 & 0.0795448& 1.26587 & 17 & 0.0770425 & 1.23154 & 18 & 0.0747616 & 1.19984 \\
19 & 0.0726714& 1.17046 & 20 & 0.0707467 & 1.14312 & 21 & 0.0689668 & 1.11760 \\
22 & 0.0673145& 1.09371 & 23 & 0.0657754 & 1.07128 & 24 & 0.0643369 & 1.05017 \\
25 & 0.0629888& 1.03026 & 26 & 0.0617218 & 1.01143 & 27 & 0.0605282 & 0.993601 \\
28 & 0.0594012& 0.976678 & 29 & 0.0583348 & 0.960589 & 30 & 0.0573239 & 0.945268 \\
31 & 0.0563636& 0.930655 & 32 & 0.0554500 & 0.916699 & 33 & 0.0545793 & 0.903351 \\
34 & 0.0537484& 0.890568 & 35 & 0.0529543 & 0.878313 & 36 & 0.0521943 & 0.866549 \\
37 & 0.0514662& 0.855244 & 38 & 0.0507676 & 0.844371 & 39 & 0.0500968 & 0.833901 \\
40 & 0.0494518& 0.823810 & 41 & 0.0488311 & 0.814077 & 42 & 0.0482332 & 0.804681 \\
43 & 0.0476567& 0.795602 & 44 & 0.0471004 & 0.786823 & 45 & 0.0465631 & 0.778329 \\
46 & 0.0460438& 0.770103 & 47 & 0.0455414 & 0.762133 & 48 & 0.0450552 & 0.754405 \\
49 & 0.0445842& 0.746907 & 50 & 0.0441276 & 0.739628 &    &          &        \\
\hline
\end{tabular*}
\caption{Values of $(a_0,T)$ for embedded examples with $l=1$, $4\leq n\leq 50$.}
\end{table}

\end{ex} 

\begin{rem} 
Numerical evidence shows that for every $k$ and $l$, there is at least one embedded example in $S^{k+l+2}$. In \cite{PerdNE1} the author numerically computes the volume of  these embeddings in $S^{n+1}$ for $n=3,\dots,13$.
\end{rem}


\bmhead{Acknowledgements}

The author would like to thank the referee for a careful reading of the manuscript and for constructive comments and suggestions that helped to improve both the exposition and the structure of the paper.

\bibliography{sn-bibliography}

\end{document}